\documentclass[final,onefignum,onetabnum]{siamart220329}
\usepackage{braket,amsfonts}

\usepackage{array}
\usepackage{bm}
\usepackage{amsmath,amssymb}
\usepackage{mathtools}
\usepackage{stmaryrd}
\usepackage{blkarray}
\usepackage[caption=false]{subfig}
\usepackage{pgfplots}

\newsiamthm{claim}{Claim}
\newsiamremark{remark}{Remark}
\newsiamremark{hypothesis}{Hypothesis}
\newsiamremark{assumption}{Assumption}
\crefname{hypothesis}{Hypothesis}{Hypotheses}

\usepackage[noend]{algpseudocode}

\usepackage{graphicx,epstopdf}

\Crefname{ALC@unique}{Line}{Lines}
\allowdisplaybreaks

\usepackage{amsopn}

\DeclareMathOperator{\diag}{diag}
\DeclareMathOperator{\rk}{rank}

\newcommand{\R}{\mathbb{R}}

\newcommand{\proj}{\mathcal{P}}
\newcommand{\bigO}{\mathcal{O}}

\newcommand{\lowrank}[2]{\llbracket {#1} \rrbracket_{#2}}
\newcommand{\norm}[1]{\left\lVert#1\right\rVert}

\usepackage[normalem]{ulem}

\newcommand{\ignore}[1]{}

\usepackage{xspace}
\usepackage{bold-extra}
\usepackage[most]{tcolorbox}

\colorlet{texcscolor}{blue!50!black}
\colorlet{texemcolor}{red!70!black}
\colorlet{texpreamble}{red!70!black}
\colorlet{codebackground}{black!25!white!25}

\lstdefinestyle{siamlatex}{%
  style=tcblatex,
  texcsstyle=*\color{texcscolor},
  texcsstyle=[2]\color{texemcolor},
  keywordstyle=[2]\color{texemcolor},
  moretexcs={cref,Cref,maketitle,mathcal,text,headers,email,url},
}

\tcbset{%
  colframe=black!75!white!75,
  coltitle=white,
  colback=codebackground, 
  colbacklower=white, 
  fonttitle=\bfseries,
  arc=0pt,outer arc=0pt,
  top=1pt,bottom=1pt,left=1mm,right=1mm,middle=1mm,boxsep=1mm,
  leftrule=0.3mm,rightrule=0.3mm,toprule=0.3mm,bottomrule=0.3mm,
  listing options={style=siamlatex}
}

\newtcblisting[use counter=example]{example}[2][]{%
  title={Example~\thetcbcounter: #2},#1}

\newtcbinputlisting[use counter=example]{\examplefile}[3][]{%
  title={Example~\thetcbcounter: #2},listing file={#3},#1}

\DeclareTotalTCBox{\code}{ v O{} }
{ 
  fontupper=\ttfamily\color{black},
  nobeforeafter,
  tcbox raise base,
  colback=codebackground,colframe=white,
  top=0pt,bottom=0pt,left=0mm,right=0mm,
  leftrule=0pt,rightrule=0pt,toprule=0mm,bottomrule=0mm,
  boxsep=0.5mm,
  #2}{#1}

\patchcmd\newpage{\vfil}{}{}{}
\begin{tcbverbatimwrite}{tmp_\jobname_header.tex}
\title{Accuracy and Stability of CUR decompositions with Oversampling
\thanks{Date: \today 
\funding{TP was supported by the Heilbronn Institute for Mathematical Research. YN is supported by EPSRC grants EP/Y010086/1 and EP/Y030990/1.}
}
}

\author{Taejun Park\thanks{Mathematical Institute, University of Oxford, Oxford, OX2 6GG, UK, (\email{park@maths.ox.ac.uk}, \email{nakatsukasa@maths.ox.ac.uk}).}
\and Yuji Nakatsukasa\footnotemark[2] }

\headers{Accuracy and Stability of CUR decompositions with Oversampling}{Taejun Park and Yuji Nakatsukasa}
\end{tcbverbatimwrite}
\input{tmp_\jobname_header.tex}

\ifpdf
\hypersetup{ pdftitle={Accuracy and Stability of CUR decompositions with Oversampling} }
\fi

\begin{document}
\maketitle

\begin{tcbverbatimwrite}{tmp_\jobname_abstract.tex}
\begin{abstract}
This work investigates the accuracy and numerical stability of CUR decompositions with oversampling. The CUR decomposition approximates a matrix using a subset of columns and rows of the matrix. When the number of columns and the rows are the same, the CUR decomposition can become unstable and less accurate due to the presence of the matrix inverse in the core matrix. Nevertheless, we demonstrate that the CUR decomposition can be implemented in a numerical stable manner and illustrate that oversampling, which increases either the number of columns or rows in the CUR decomposition, can enhance its accuracy and stability. Additionally, this work devises an algorithm for oversampling motivated by the theory of the CUR decomposition and the cosine-sine decomposition, whose competitiveness is illustrated through experiments. 
\end{abstract}


\end{tcbverbatimwrite}
\input{tmp_\jobname_abstract.tex}

\section{Introduction} \label{sec:intro}
The computation of a low-rank approximation to a matrix is omnipresent in the computational sciences \cite{UdellTownsend2019}. It has seen increasing popularity due to its importance in tackling large scale problems. An important aspect of low-rank approximation is the selection of bases that approximate the span of a matrix's column and/or row spaces. In this work, we consider a natural choice of using a subset of rows and columns of the original matrix as low-rank bases, namely the CUR decomposition. The CUR decomposition \cite{BoutsidisWoodruff2017,MahoneyDrineas2009,SorensenEmbree2016}, also known as a ``matrix skeleton'' approximation \cite{GoreinovTyrtyshinikovZamarashkin1997}, is a low-rank approximation that approximates a matrix $A$ as a product\footnote{The notation $Z$ is temporarily used to denote the core matrix instead of the traditional notation $U$ as we use $U$ to denote the intersection of $C$ and $R$, i.e., $U = A(I,J)$ later.}
\begin{equation}
    \underset{m\times n}{A} \approx \underset{m\times k}{C} \; \underset{k\times k}{Z} \;\underset{k \times n}{R}
\end{equation}
 where in MATLAB notation, $C = A(:,J)$ is a $k$-subset of columns of $A$ with $J\subseteq \{1,\ldots,m\}$ being the column indices and $R = A(I,:)$ is a $k$-subset of the rows of $A$ with $I\subseteq \{1,\ldots,n\}$ being the row indices. The factors $C$ and $R$ are subsets of the original matrix $A$ that inherit certain properties of the original matrix, such as sparsity or nonnegativity, a property that is absent in the truncated SVD. They also assist with data interpretation by revealing the important rows and columns of $A$. The CUR decomposition is also memory efficient when the entries of $A$ can be computed or extracted quickly because we can just store the column and row indices without explicitly storing $C$ and $R$. 

There are two common choices for $Z$, which we refer to as the core matrix: $C^\dagger A R^\dagger$ or $A(I,J)^{-1}$ (or more generally $A(I,J)^\dagger$) \cite{HammHuang2020}. The choice $Z = C^\dagger AR^\dagger$ minimizes the Frobenius norm error 
$\|A-CZR\|_F$
given the choice of $C$ and $R$. Hence, we refer to the CUR decomposition with this choice as CURBA (CUR with Best Approximation) for shorthand in this work. While this choice is more robust, it involves the full matrix $A$ to form $C^\dagger A R^\dagger$, which can be costly, requiring $O(mnk)$ operations for a dense matrix $A$. On the other hand, the choice $Z = A(I,J)^{-1}$, which is often called the cross approximation \cite{CortinovisKressner2020}, is more efficient as we only require the overlapping entries of $C$ and $R$, and need not even read the whole matrix $A$. We refer to this version of the CUR decomposition as CURCA (CUR with Cross Approximation) for shorthand. However, this choice has the drawback that $A(I,J)$ can be (nearly) singular, which can lead to poor approximation error. Indeed, the ill-conditioning of $A(I,J)$ for CURCA becomes alarming when computing its matrix (pseudo)inverse, a concern that has been raised in various papers \cite{DongMartinsson2023,MartinssonTropp2020,SorensenEmbree2016}. 
For this reason, the related interpolative decomposition~\cite{martinsson2019randomized} is often advocated to avoid numerical instability with CUR, especially CURCA.
The aim of this paper is to address the instability of the CURCA by demonstrating that it can be implemented in a numerically stable manner using the $\epsilon$-pseudoinverse in the presence of roundoff errors. The $\epsilon$-pseudoinverse variant takes the $\epsilon$-pseudoinverse of $A(I,J)$ by first truncating the singular values of $A(I,J)$ less than $\epsilon$ before computing its pseudoinverse, which has been explored in a related context in \cite{CaiNagyXi2022,ChiuDemanet2013,Nakatsukasa2020,NakatsukasaPark2023Indef}. We abbreviate the CURCA with $\epsilon$-pseudoinverse as the stabilized CURCA (SCURCA for short). Furthermore, we illustrate that oversampling, which increases either the number of columns or rows in the CURCA, can be incorporated to enhance its accuracy and stability. Specifically, we recommend oversampling in proportion to the target rank; see Section \ref{subsec:oversampleNI}. Throughout this work, we concentrate on CURCA and take $U = A(I,J)$ unless otherwise stated, such that CURCA can be written as $CU^\dagger R$, but the analysis and the relevant counterparts regarding CURBA can be found in Appendix \ref{app:CURBA}.\footnote{For the CURBA, there already exists a numerically stable way of computing the CURBA given by the StableCUR algorithm in \cite{AndersonDuMahoneyMelgaardWuGu2015}. Furthermore, the effect of oversampling is less effective as the CURBA, $A\approx C(C^\dagger A R^\dagger) R$ is a more robust and stable (but expensive) approximation than the CURCA, $A \approx CU^\dagger R$; see Appendix \ref{app:CURBA}.}

\paragraph{Overview} In this work, we focus on the CURCA. Section \ref{sec:relatedwork} provides background information on the CUR decomposition, discusses related and existing works and outlines our contributions. In Section \ref{sec:analysis}, we prove a theoretical bound for the CURCA and its $\epsilon$-pseudoinverse variant, the SCURCA with oversampling and show that the SCURCA can be computed in a numerically stable way in the presence of roundoff errors.\footnote{The stability analysis of the plain CURCA, $CU^\dagger R$ is not covered in this work; in our analysis we require using the $\epsilon$-pseudoinverse. 
However, we observe its numerical stability in practice; see Section \ref{subsec:implementation}.} In Section \ref{sec:method}, we study oversampling of indices for the CURCA with the purpose of improving its accuracy and numerical stability based on the theory developed in Section \ref{sec:analysis}. The setting we are interested in is when we have row indices $I$ and column indices $J$ with $|I| = |J| = k$, how should we oversample one of $I$ or $J$. We address this question by proposing a deterministic way of oversampling for which the rationale for our idea can be explained by the cosine-sine (CS) decomposition. 
The construction could be of use in other contexts where (over)sampling from a subspace is desired; a common task that finds use, for example, in model reduction \cite{ChaturantabutSorenson2010} and active learning~\cite{shimizu2023improved}.
We conclude with numerical illustrations in Section \ref{sec:numill}, demonstrating the stability of our method in computing the CURCA and the strength of oversampling.

\paragraph{Notation} 
Throughout, we use $\norm{\cdot}_2$ for the spectral norm or the vector-$\ell_2$ norm, $\norm{\cdot}_F$ for the Frobenius norm and $\norm{\cdot}$ for any unitarily invariant norms. We use dagger $^\dagger$ to denote the pseudoinverse of a matrix and $\lowrank{A}{k}$ to denote the best rank-$k$ approximation to $A$ in any unitarily invariant norm, i.e., the approximation derived from truncated SVD~\cite{hornjohnson}. Unless specified otherwise, $\sigma_i(A)$ denotes the $i$th largest singular value of the matrix $A$. We use MATLAB style notation for matrices and vectors. For example, for the $k$th to $(k+j)$th columns of a matrix $A$ we write $A(:,k:k+j)$. We use $I$ and $J$ for the row and the column indices respectively and set $\Pi_I = I_m(:,I)$ and $\Pi_J = I_n(:,J)$ so that $A(:,J) = A\Pi_J$ and $A(I,:) = \Pi_I^T A$ for $A\in \R^{m\times n}$. Here, $I_m$ denotes the $m\times m$ identity matrix. Lastly, we use $|I|$ to denote the cardinality of the index set $I$ and define $[n]:=\{1,2,...,n\}$.


\section{Background, Related Work and Contributions} \label{sec:relatedwork}
In the CUR decomposition, it is important to get a good set of row and column indices as they dictate the quality of the low-rank approximation. There are many practical methods, but they largely fall into two different categories: pivoting or sampling. For pivoting based methods, we use pivoting schemes such as column pivoted QR (CPQR) \cite{GolubVanLoan2013} or LU with complete pivoting \cite{TrefethenBau} on $A$ or the singular vectors of $A$ to obtain the pivots which we then use as row or column indices. For example, these can be used on the dominant singular vectors of $A$ to obtain the pivots and the discrete empirical interpolation method (DEIM) is a popular example \cite{ChaturantabutSorenson2010,DrmacGugercin2016,SorensenEmbree2016}. There are also pivoting schemes that provide a strong theoretical guarantee \cite{GuEisenstat1996}. Applying pivoting schemes directly on $A$ can be prohibitive for large matrices, and for this reason a number of randomized algorithms based on \emph{sketching} have been proposed \cite{DongMartinsson2023,Duersch_Gu_SIREV,VoroninMartinsson2017}. These methods ``sketch'' the original matrix down to a smaller-sized matrix using randomization and perform the pivoting schemes there. See \cite{DongMartinsson2023} for a comparative study of randomized pivoting algorithms. On the other hand, for sampling based methods, we sample the column or row indices from some probability distribution obtained from certain information about $A$. For example, a popular choice is the row norms of the dominant singular vectors of $A$ for leverage scores \cite{DrineasMahoneyMuthukrishnan2008,MahoneyDrineas2009}. There are other sampling strategies such as uniform sampling \cite{ChiuDemanet2013}, volume sampling \cite{CortinovisKressner2020,DeshpandeRademacherVempalaWang2006,DeshpandeVempala2006,GoreinovTyrtyshinikovZamarashkin1997}, DPP sampling \cite{DerezinskiMahoney2021}, and BSS sampling \cite{BatsonSpielmanSrivastava2012,BoutsidisDrineasMagdon-Ismail2014}. In particular, volume sampling leads to a CUR approximation that have close-to-optimal error guarantees \cite{CortinovisKressner2020,ZamarashkinOsinsky2018}. Sampling based methods can also be prohibitive for large matrices, so a smaller-sized proxy of $A$ via sketching have been proposed \cite{DrineasMagdon-IsmailMahoneyWoodruff2012}. There is also a deterministic sampling method for leverage scores \cite{PapailiopoulosKyrillidis2014} and hybrid methods such as L-DEIM \cite{GidisuHochstenbach2022}.

The vast majority of CUR decompositions of $A$ comes with a theoretical guarantee that involves the term
\begin{equation} \label{eq:NBterm}
    \norm{V(J,:)^{-1}}_2 = \frac{1}{\sigma_{\min}(V(J,:))},
\end{equation} where $V\in \R^{n\times k}$ is the $k$ (approximate) dominant right singular vectors of $A \in \R^{m\times n}$ and $J \subset [n]$ with $|J| = k$ is the set of column indices; see Theorems \ref{thm:CAe} and \ref{thm:CURBA}. A similar term involving the left singular vectors and a row index set is also present. The term \eqref{eq:NBterm} is usually the deciding factor for the accuracy of the CUR decomposition and therefore, a majority of the algorithms aim at choosing a set of indices $J$ that would diminish the effect of \eqref{eq:NBterm}. A natural way of improving \eqref{eq:NBterm} is to \emph{oversample}, that is, obtain extra indices $J_0 \in \R^p$ distinct from $J$ so that $V(J\cup J_0,:) \in \R^{(k+p)\times k}$ becomes a rectangular matrix. By appending more rows to $V(J,:)$, $V(J\cup J_0,:)$ has a larger minimum singular value by the Courant-Fischer min-max theorem, which improves the accuracy of the CUR decomposition, as we will see in the later sections. The topic of oversampling is not new, in particular, it is shown in \cite{AndersonDuMahoneyMelgaardWuGu2015} that oversampling improves the accuracy of CURBA when the singular values decay rapidly. In the context of sampling based methods, oversampling has been suggested for theoretical guarantees; see e.g. \cite{BoutsidisWoodruff2017,ChiuDemanet2013,MahoneyDrineas2009}. It is easy to oversample for sampling based methods as we can simply sample more than $k$ indices from the given probability distribution. In contrast, it is often difficult to oversample for pivoting based methods. This is because we typically perform pivoting schemes on a smaller-sized surrogate $X\in \R^{n\times k}$, for example, the sketch of $A$ and the pivots beyond the first $k$ indices carry little to no information. Therefore, for pivoting based methods, usually a different strategy is used for oversampling. In the context of DEIM, various ways of oversampling have been suggested \cite{CarlbergFarhatCortialAmsallem2013,GidisuHochstenbach2022,PeherstorferDrmacGugercin2020,ZimmermannWillcox2016}. Specifically, Zimmermann and Willcox \cite{ZimmermannWillcox2016} show that oversampling can improve the condition number of oblique projections, and Donello et al. \cite{DonelloPalkarNaderiDelReyFernandezBabaee2023} proves a bound for DEIM projectors\footnote{DEIM projector is given by $\mathcal{P}_U = U(\Pi^T U)^\dagger \Pi^T$ where $U\in \R^{n\times k}$ is an orthonormal matrix and $\Pi \in \R^{n\times k}$ is a submatrix of the identity matrix that picks $k$ chosen rows.} with oversampling, which extends the proof without oversampling from \cite{SorensenEmbree2016}. A majority of the aforementioned literature focuses on the DEIM projector or the CURBA. However, oversampling can often be more effective for the CURCA as we not only improve the bound involving \eqref{eq:NBterm} (see Theorem \ref{thm:CAe}), we also make $A(I,J)$ a rectangular matrix when we oversample either $I$ or $J$ (but not both), which improves the condition number of $A(I,J)$ making the CURCA more accurate when computing the pseudoinverse of $A(I,J)$. Therefore, we advocate oversampling when possible over the standard choice $|I| = |J|$.

\paragraph{Existing work}
The numerical stability of the stabilized CURCA, $A\approx CU_{\epsilon}^\dagger R$, involving the $\epsilon$-pseudoinverse has not been studied previously to our knowledge. However, there are some related works exploring the idea of $\epsilon$-pseudoinverse in the core matrix. The $\epsilon$-pseudoinverse takes the core matrix $U = A(I,J)$, truncates its singular values that are less than $\epsilon$, and computes the pseudoinverse of the resulting matrix. Firstly, Chiu and Demanet \cite{ChiuDemanet2013} study the $\epsilon$-pseudoinverse in the context of the CURCA with uniform sampling and show that when the matrix is incoherent, the algorithm succeeds. However, this paper does not contain stability analysis, and has the condition that the matrix needs to be incoherent. The authors in \cite{CaiNagyXi2022,NakatsukasaPark2023Indef} explore the $\epsilon$-pseudoinverse in the context of the symmetric Nystr\"om method, $A\approx CU^\dagger_\epsilon C^T$, applied to symmetric indefinite matrices. However, they show that the $\epsilon$-pseudoinverse can deteriorate the accuracy of the symmetric Nystr\"om method when applied to symmetric indefinite matrices. Lastly, Nakatsukasa \cite{Nakatsukasa2020} studies the generalized Nystr\"om algorithm\footnote{The generalized Nystr\"om method is a variant of the CURCA where instead of subsets of rows and columns of $A$ that approximate the row and column space of $A$, we have the sketches, $Y^TA$ and $AX$ that approximate the row and column space of $A$ where $X$ and $Y$ are random embeddings. See \cite{Nakatsukasa2020,TroppYurtseverUdellCevher2017GenNys} for more details.} with $\epsilon$-pseudoinverse and oversampling, which provides a numerically stable way of computing the generalized Nystr\"om method and proves its stability. The paper also demonstrates that oversampling is necessary for a stable approximation in the generalized Nystr\"om method. 

In a related work, Hamm and Huang study the stability of sampling for CUR decompositions in \cite{HammHuang2020StabCUR} and the perturbation bounds of CUR decompositions in \cite{HammHuang2021PerturbCUR}. In \cite{HammHuang2020StabCUR}, they study the problem of determining when a column submatrix of a rank $k$ matrix $A$ also has rank $k$. This is important as if the chosen columns or rows of a matrix is (numerically) rank-deficient then $U$ is also (numerically) rank-deficient, which may cause numerical issues when computing the pseudoinverse of $U$. In \cite{HammHuang2021PerturbCUR}, they derive perturbation bounds for CUR decompositions under the influence of a noise matrix. They investigate several variants of the CUR decomposition including the $\epsilon$-pseudoinverse variant, SCURCA, and provide perturbation bounds in terms of the noise matrix. This work is related, but different from our work as they investigate the accuracy of the CUR decomposition for the perturbed matrix $\tilde{A} = A+E$, while we investigate the accuracy and stability for a numerical implementation of the CURCA in the presence of rounding errors.

We review two existing algorithms for oversampling. First, Gidisu and Hochstenbach \cite{GidisuHochstenbach2022} oversample $p$ extra indices by choosing the largest $p$ leverage scores (row norms of (approximate) dominant singular vectors) out of the unchosen indices. The complexity is $\bigO(nk)$ for computing the leverage scores of an orthonormal matrix $V\in \R^{n\times k}$. If we are only given an approximator that is not orthonormal then (approximate) orthonormalization needs to be done, usually at a cost of $\bigO(nk^2)$. Second, Peherstorfer, Drma\v{c} and Gugercin \cite{PeherstorferDrmacGugercin2020}, iteratively select $p$ extra indices for oversampling to maximize the minimum singular value of $V(J,:)$ in a greedy fashion. This approach, which is called the $\mathtt{GappyPOD\!+\!E}$ algorithm, uses perturbation bounds on the eigenvalues given in \cite{IpsenNadler2009} to find the next index that maximizes the lower bound for the minimum singular value of $V(J,:)$. This approach is also a special case of \cite{ZimmermannWillcox2016}. The algorithm runs with complexity $\bigO((k+p)^2k^2+nk^2p)$ where $p$ is the number of indices we oversample by. Again, if $V$ is not orthonormal to begin with, then (approximate) orthonormalization needs to be done usually at a cost of $\bigO(nk^2)$. For a treatment of approximate orthonormalization, see, for example \cite{avron10,BalabanovGrigori2022}.

\paragraph{Contributions}
Our contribution is twofold: (1) presenting a method for computing the CURCA in a numerically stable manner with a theoretical guarantee, and (2) advocating the use of oversampling to improve the accuracy and stability of the CURCA.

Our first and main contribution lies in presenting a method for computing the CURCA in a numerically stable manner, accompanied by an analysis that guarantees its stability. We show that with the $\epsilon$-pseudoinverse in the core matrix, the SCURCA, $A\approx CU_{\epsilon}^\dagger R$, can be computed in a numerically stable manner by taking the following steps. First, we compute each row of $C U_{\epsilon}^\dagger$ using a backward stable underdetermined linear solver. Then we compute the SCURCA by multiplying $CU_{\epsilon}^\dagger$ by $R$. See Section \ref{subsec:numstab} for details. In addition to the stability analysis, we also analyze the CURCA and its $\epsilon$-pseudoinverse variant, SCURCA in exact arithmetic by deriving a relative norm bound. While our analysis does not cover the stability of plain CURCA, $CU^\dagger R$, we observe its stability in practice without the $\epsilon$-truncation; see Sections \ref{subsec:numstab} and \ref{subsec:implementation}.

Our secondary contribution involves advocating the use of oversampling for the CURCA and for providing a deterministic algorithm to oversample row or column indices. We show that oversampling improves the accuracy and stability of the CURCA by providing a theoretical analysis that demonstrates the benefits of oversampling. We show that oversampling should be done such that it increases the minimum singular value(s) of $V(J,:)$ where $V\in \R^{n\times k}$ is the $k$ (approximate) dominant right singular vectors of $A$ and $J$ is a set of indices with $|J| = k$. Our algorithm is motivated by the cosine-sine (CS) decomposition and runs with complexity $\bigO(nk^2+nkp)$ 
where $k$ is the target rank and $p$ is the oversampling parameter. Note that this complexity only refers to the cost of the oversampling process, not the whole CUR process; the complexity of the initial CUR process is separate and depends on the method used to get the initial set of indices.
We show that our algorithm is competitive with existing algorithms and in particular, performs similarly to the $\mathtt{GappyPOD\!+\!E}$ algorithm in \cite{PeherstorferDrmacGugercin2020}, in which the oversampling process runs with complexity $\bigO((k+p)^2k^2+nk^2p)$. The numerical experiments illustrate that oversampling is recommended.

\section{Accuracy and stability of the stabilized CURCA} \label{sec:analysis}
In this section, we study two topics related to the CURCA, $A\approx CU^\dagger R$. We first analyze the accuracy of the CURCA, 
\begin{equation*}
    A_{IJ} = A(:,J)A(I,J)^\dagger A(I,:) = A\Pi_J (\Pi_I^TA\Pi_J)^\dagger \Pi_I^T A =: C U^\dagger R
\end{equation*}
and its $\epsilon$-pseudoinverse variant, the SCURCA,
\begin{equation*}
    A_{IJ}^{\epsilon} = A(:,J) A(I,J)_{\epsilon}^\dagger A(I,:) = A\Pi_J (\Pi_I^TA\Pi_J)_\epsilon^\dagger \Pi_I^T A = CU_\epsilon^\dagger R,
\end{equation*} and show that the $\epsilon$-truncation in the core matrix compromises the accuracy of the CURCA only by $\epsilon$ times the condition number of the CURCA; see Remark \ref{remark:CURCA}. We then analyze the numerical stability of SCURCA in the presence of roundoff errors and show that SCURCA satisfies a similar bound under roundoff errors, making the SCURCA numerically stable as long as the selected rows and columns well-approximate the dominant row and column spaces of $A$; see Section \ref{subsec:numstab}. 

We begin with some preliminaries: oblique projectors and standard assumptions. In the proofs below, we frequently use oblique projectors and their properties. We use $\proj_{X,Y} := X(Y^TX)^\dagger Y^T$ where $X\in \R^{n\times k}$ and $Y\in \R^{n\times \ell}$ to denote an oblique projection onto the column space of $X$ if $k\leq \ell$ and $Y^TX$ has full column rank or onto the row space of $Y^T$ if $\ell \leq k$ and $Y^TX$ has full row rank. For example, the CUR decomposition $A_{IJ}$ can be written as 
\begin{equation*}
    A_{IJ} = A\Pi_J (\Pi_I^TA\Pi_J)^\dagger \Pi_I^T A = \proj_{A\Pi_J,\Pi_I} A = A\proj_{\Pi_J,A^T\Pi_I}.
\end{equation*} Some of the important properties of projectors \cite{Szyld2006} are 
\begin{enumerate}
    \item $\proj_{X,Y}\proj_{X,Y} = \proj_{X,Y}$,
    \item $\proj_{X,Y}X = X$ if $Y^TX$ has full column rank,
    \item $Y^T\proj_{X,Y} = Y^T$ if $Y^TX$ has full row rank,
    \item $\norm{\proj_{X,Y}}_2 = \norm{I-\proj_{X,Y}}_2$ if $\proj_{X,Y} \neq 0, I$.
\end{enumerate} Lastly, sometimes we can simplify the norm of oblique projectors, which is given by the lemma below.
\begin{lemma} \label{lemma:obproj}
    Let $\proj_{X,Y} \in \R^{n\times n}$ be a projector where $X\in \R^{n\times k}$, $Y\in \R^{n\times \ell}$ and $Y^TX\in \R^{\ell\times k}$ all have full column rank (so $k\leq \ell$). Then
    \begin{equation}
        \norm{\proj_{X,Y}} = \norm{(Y^T Q_X)^\dagger Y^T}
    \end{equation} for any unitarily invariant norm $\norm{\cdot}$ where $Q_X$ is an orthonormal matrix spanning the columns of $X$.
\end{lemma}
\begin{proof}
    Let $X = Q_X R_X$ be the thin QR decomposition of $X$. Then
    \begin{align*}
       \norm{\proj_{X,Y}} = \norm{X(Y^TX)^\dagger Y^T} = \norm{Q_XR_X(Y^TQ_XR_X)^\dagger Y^T} = \norm{(Y^TQ_X)^\dagger Y^T},
    \end{align*} since $Y^TQ_X\in \R^{\ell \times k}$ has full column rank and $R_X \in \R^{k\times k}$ is nonsingular as $Y^TX$ has full column rank.
\end{proof}

Now, we lay out some generic assumptions that hold in our theorems below. The assumptions are
\begin{assumption} \label{assumptions} \*
\begin{enumerate}
    \item $|I|=|J| = k \leq \rk(A)$ where $k$ is the target rank,
    \item $A(I,J)\in\R^{k\times k}$ is a non-singular matrix,
    \item $X(:,J) \in \R^{k\times k}$ has full row rank, where $X$ is (any) row space approximator of $A$.\footnote{The assumptions and the theorems in this section are stated in terms of row space approximators, but similar assumptions and theorems for column space approximators can be obtained, for example, by considering $A^T$ instead.}
\end{enumerate}
\end{assumption}
When $\rk(A)\leq k$ and $\rk(A(I,J)) = \rk(A)$, we have $A = A_{IJ}$ \cite{HammHuang2020}. Since $A(I,J)$ is assumed to be non-singular, $A(:,J)$ and $A(I,:)$ have full column and row rank, respectively.
Under Assumption \ref{assumptions}, by Lemma \ref{lemma:obproj},
\begin{equation*}
    \norm{C U^\dagger} =  \norm{Q_C(I,:)^\dagger}, \norm{U^\dagger R} = \norm{Q_R(J,:)^\dagger}, 
    \norm{X(J,:)^\dagger X} = \norm{Q_X(J,:)^\dagger}
\end{equation*} where $Q_C$, $Q_R$ and $Q_X$ are the orthonormal matrices spanning the columns of $C$, $R^T$ and $X^T$, respectively. We now prove the accuracy of the CURCA and its $\epsilon$-pseudoinverse variant, the SCURCA.

\subsection{Accuracy of CUR and its $\epsilon$-pseudoinverse variant} \label{subsec:CURanal}
We prove the accuracy of the SCURCA, $A_{IJ}^\epsilon$ first. The accuracy for the CURCA, $A_{IJ}$ follows by setting $\epsilon = 0$. The analysis presented in this section is a key contribution and plays an essential role for the stability analysis as we show that the SCURCA satisfies a similar bound under roundoff errors, establishing its numerical stability. The bound in Theorem \ref{thm:CAe} below somewhat resembles that in \cite{DrineasIpsen2019}.

\begin{theorem} \label{thm:CAe}
    Let $A\in \R^{m\times n}$ be a matrix, $I$ and $J$ be a set of row and column indices, respectively, with $|I| = |J| = k$, $\epsilon >0$ and $X \in \R^{k\times n}$ be any row space approximator of $A$. Then under Assumption \ref{assumptions},
    \begin{equation} \label{eq:OSThm}
         \norm{A-A_{I\cup I_0,J}^{\epsilon}} \leq \norm{Q_C(I\cup I_0,:)^\dagger}_2 \norm{Q_X(J,:)^{-1}}_2 \left(\norm{A-AX^\dagger X} + \norm{E}\right)
    \end{equation} for any unitarily invariant norm $\norm{\cdot}$ where $I_0$ is a set of extra row indices distinct from $I$ with $|I_0| = p$, and $E\in \R^{k_\epsilon \times k_\epsilon }$ is a matrix satisfying $\norm{E}_2 \leq \epsilon$ where $k_\epsilon \leq k$ is the number of singular values of $A(I\cup I_0,J)$ smaller than $\epsilon$.
\end{theorem}
\begin{proof}
    For shorthand, let $I_* :=I\cup I_0$. Let the thin SVD of $A(I_*,J)\in \R^{(k+p)\times k} $ be 
        $W \Sigma V^T = [W_1,W_2] \diag(\Sigma_1,\Sigma_2) [V_1,V_2]^T$ where $\Sigma_2$ contains the singular values of $A(I_*,J)$ smaller than $\epsilon$. Then
        \begin{align*}
            A_{I_*J}^\epsilon \Pi_J &= A\Pi_J (\Pi_{I_*}^TA\Pi_J)_\epsilon^\dagger \Pi_{I_*}^TA\Pi_J 
            = A\Pi_J V_1 \Sigma_1^{-1} W_1^T W\Sigma V^T \\
            &= A\Pi_J V_1V_1^T 
            = A\Pi_J - A\Pi_J V_2 V_2^T.
        \end{align*}
    Therefore,
    \begin{align} \label{eq:OSthmErr}
        A-A_{I_* J}^\epsilon &= \left(I-A\Pi_J\left(\Pi_{I_*}^TA\Pi_J\right)_\epsilon^\dagger \Pi_{I_*}^T\right)A \notag \\
        &= (I-A\Pi_J(\Pi_{I_*}^TA\Pi_J)_\epsilon^\dagger \Pi_{I_*}^T)A(I-\Pi_J(X\Pi_J)^\dagger X) +A\Pi_J V_2V_2^T(X\Pi_J)^\dagger X.
    \end{align} 
    Note that $\proj_{A\Pi_J,\Pi_{I_*}}^\epsilon := A\Pi_J(\Pi_{I_*}^TA\Pi_J)_\epsilon^\dagger \Pi_{I_*}^T$ is an oblique projector since
    \begin{align*}
(\proj_{A\Pi_J,\Pi_{I_*}}^\epsilon)^2
 &= A\Pi_J V_1 \Sigma_1^{-1} W_1^T W\Sigma V^T V_1 \Sigma_1^{-1} W_1^T \Pi_{I_*}^T \\
        &= A\Pi_J V_1 \Sigma_1^{-1} W_1^T \Pi_{I_*}^T \\
        &= \proj_{A\Pi_J,\Pi_{I_*}}^\epsilon
    \end{align*} and similarly, $\proj_{\Pi_J,X^T} = \Pi_J(X\Pi_J)^\dagger X$ is an oblique projector.
    Now bounding the first term of \eqref{eq:OSthmErr} gives 
    \begin{align*}
        \norm{(I-\proj_{A\Pi_J,\Pi_{I_*}}^\epsilon)A(I-\proj_{\Pi_J,X^T})} &\leq \norm{I-\proj_{A\Pi_J,\Pi_{I_*}}^\epsilon}_2 \norm{A(I-\proj_{\Pi_J,X^T})} \\
        &= \norm{\proj_{A\Pi_J,\Pi_{I_*}}^\epsilon}_2 \norm{A(I-X^\dagger X)(I-\proj_{\Pi_J,X^T})} \\
        &\leq \norm{\proj_{A\Pi_J,\Pi_{I_*}}^\epsilon}_2 \norm{A(I-X^\dagger X)} \norm{I-\proj_{\Pi_J,X^T}}_2 \\
        &= \norm{\proj_{A\Pi_J,\Pi_{I_*}}^\epsilon}_2 \norm{\proj_{\Pi_J,X^T}}_2\norm{A(I-X^\dagger X)} 
    \end{align*} where in the first equality we used $X\proj_{\Pi_J,X^T} = X\Pi_J (X\Pi_J)^\dagger  = I$, as $X\Pi_J = X(:,J)$ has full row rank by Assumption \ref{assumptions}. The first term $\norm{\proj_{A\Pi_J,\Pi_{I_*}}^\epsilon}_2$ in the final expression can be bounded by letting $A\Pi_J = Q_CR_C$ be the thin QR decomposition and noting that $\Pi_{I_*}^TQ_C$ has full column rank, as
    \begin{align*}
        \norm{\proj_{A\Pi_J,\Pi_{I_*}}^\epsilon}_2 &= \norm{Q_C R_C(\Pi_{I_*}^TA\Pi_J)_\epsilon^\dagger}_2 = \norm{R_C (\Pi_{I_*}^TA\Pi_J)_\epsilon^\dagger}_2 \\
        &= \norm{(\Pi_{I_*}^TQ_C)^\dagger \Pi_{I_*}^TQ_C R_C (\Pi_{I_*}^TA\Pi_J)_\epsilon^\dagger}_2 \\
        &\leq  \norm{(\Pi_{I_*}^TQ_C)^\dagger}_2\norm{\Pi_{I_*}^TA\Pi_J (\Pi_{I_*}^TA\Pi_J)_\epsilon^\dagger}_2 \\
        &\leq \norm{(\Pi_{I_*}^TQ_C)^\dagger}_2.
    \end{align*}
    Therefore, using Lemma \ref{lemma:obproj} on $\proj_{\Pi_J,X^T}$, the first term in \eqref{eq:OSthmErr} can be bounded as
    \begin{equation*}
         \norm{(I-\proj_{A\Pi_J,\Pi_{I_*}}^\epsilon)A(I-\proj_{\Pi_J,X^T})} \leq \norm{Q_C(I_*,:)^\dagger}_2 \norm{Q_X(J,:)^{-1}}_2 \norm{A(I-X^\dagger X)}.
    \end{equation*}
    The second term in \eqref{eq:OSthmErr} can be bounded using a similar argument as 
    \begin{align*}
        \norm{A\Pi_J V_2V_2^T(X\Pi_J)^\dagger X} &= \norm{Q_CR_C V_2V_2^T (X\Pi_J)^\dagger X} 
        = \norm{R_C V_2V_2^T (X\Pi_J)^\dagger X} \\
        &= \norm{(\Pi_{I_*}^TQ_C)^\dagger \Pi_{I_*}^TQ_C R_C V_2V_2^T (X\Pi_J)^\dagger X} \\
        &= \norm{(\Pi_{I_*}^TQ_C)^\dagger W_2\Sigma_2V_2^T (X\Pi_J)^\dagger X} \\
        &\leq \norm{(\Pi_{I_*}^TQ_C)^\dagger}_2 \norm{W_2\Sigma_2V_2^T} \norm{(X\Pi_J)^\dagger X}_2 \\
        &=\norm{Q_C(I_*,:)^\dagger}_2 \norm{Q_X(J,:)^{-1}}_2 \norm{\Sigma_2}.
    \end{align*}
    Putting everything together and letting $E = \Sigma_2$, we get the desired result.
\end{proof}
\begin{corollary} \label{cor:CAe}
    Under the same assumptions as in Theorem \ref{thm:CAe},
    \begin{equation} \label{eq:OSCor}
         \norm{A-A_{I\cup I_0,J}} \leq \norm{Q_C(I\cup I_0,:)^\dagger}_2 \norm{Q_X(J,:)^{-1}}_2 \norm{A-AX^\dagger X}
    \end{equation} for any unitarily invariant norm $\norm{\cdot}$.
\end{corollary}
\begin{proof}
    Set $\epsilon = 0$ in Theorem \ref{thm:CAe}.
\end{proof}

\begin{remark}\* \label{remark:CURCA}
\begin{enumerate}
    \item The condition number of the CURCA is $\kappa=\norm{Q_C(I\cup I_0,:)^\dagger}_2 \norm{Q_X(J,:)^{-1}}_2$, as indicated by~\eqref{eq:OSThm}. 
Theorem \ref{thm:CAe} tells us that the SCURCA, $A_{I_*J}^\epsilon$ is 
worse than the CURCA by at most a factor $\kappa\sqrt{k}\epsilon$ in the Frobenius norm.
    \item The bound in Theorem \ref{thm:CAe} and Corollary \ref{cor:CAe} has two factors involving the (pseudo)inverse, which is in contrast to the CURBA, $C(C^\dagger A R^\dagger) R$ (see Appendix \ref{app:CURBA}) having only one factor. This makes the CURCA, $A_{IJ}$ usually worse than the CURBA, which is expected as the CURCA is cheaper to compute. However, the CURCA can still be very accurate; see for example \cite{CortinovisKressner2020,ZamarashkinOsinsky2018}, which establishes the existence of a rank-$r$ CURCA that has error within a factor $r+1$ of the best rank-$r$ approximation via the truncated SVD. The second multiplicative factor in Theorem \ref{thm:CAe} and Corollary \ref{cor:CAe} comes from the fact that $A_{IJ}$ is associated with the oblique projector $\proj_{A\Pi_J,\Pi_I}$ rather than the orthogonal projectors $CC^\dagger$ and $R^\dagger R$ for the CURBA. Nonetheless, the CURCA and the CURBA have comparable accuracy when employed with good row and column indices and oversampling, with the CURCA being much more computationally efficient.
    \item The row space approximator $X\in \R^{k\times n}$ can be chosen in various ways, for example, the $k$-dominant right singular vectors of $A$ or the row sketch of $A$. The bounds in Theorem \ref{thm:CAe} and Corollary \ref{cor:CAe} can both be computed a posteriori when $X$ can be computed easily. The $k$-dominant right singular vectors would make the right-most term, $\norm{A-AX^\dagger X}$ in the bound of Theorem \ref{thm:CAe} and Corollary \ref{cor:CAe}, optimal. However, the singular vectors are often too expensive to compute.
\end{enumerate}
\end{remark}

Theorem \ref{thm:CAe} and Corollary \ref{cor:CAe} provide a bound for the SCURCA and the CURCA, respectively. It also demonstrates the benefit of oversampling through the extra set of row indices $I_0$. We obtain a pseudoinverse for $\norm{Q_C(I\cup I_0,:)^\dagger}_2$ instead of the matrix inverse $\norm{Q_C(I,:)^{-1}}_2$ with the former always being smaller. Additionally, since $\norm{Q_C(I\cup I_0,:)^\dagger}_2 = \norm{A(:,J)A(I_*,J)^\dagger}_2$ and we want to minimize this quantity, the row indices $I$ and $I_0$ should be chosen in terms of the already-chosen columns of $A$. This has been suggested and employed in other works such as \cite{CortinovisYing2024,DongMartinsson2023,VoroninMartinsson2017,Xia2024}. This comes with the benefit that the resulting core matrix $A(I_*,J)$ will generally be better-conditioned, improving the accuracy of the CURCA. For example, consider the following $2\times 2$ matrix, 
\begin{equation*}
    A = \begin{bmatrix}
        \epsilon & 1 \\
        1 & 0
    \end{bmatrix}
\end{equation*} where $0<\epsilon<1$. For a rank-$1$ approximation of $A$, we need to choose a column and a row for the CURCA. If we choose a column and a row separately in the best possible way, we would choose the first row and the first column, giving us
\begin{equation*}
    A_{1,1} = \begin{bmatrix}
        \epsilon \\ 1
    \end{bmatrix} \epsilon^{-1} \begin{bmatrix}
        \epsilon & 1
    \end{bmatrix} = \begin{bmatrix}
        \epsilon & 1 \\
        1 & 1/\epsilon
    \end{bmatrix},
\end{equation*}which is a poor approximation as $\norm{A-A_{1,1}}_F = 1/\epsilon$ can be arbitrary large\footnote{The accuracy of CURBA, by contrast, is good as long as the chosen columns and rows are good approximators for the range and co-range of $A$~\cite[Remark 1]{DongMartinsson2023}.} as $\epsilon \rightarrow 0$. On the other hand, if we choose a column first and then a row, we choose the first column $[\epsilon,1]^T$ and the second row as $\epsilon<1$, giving us
\begin{equation*}
    A_{2,1} = \begin{bmatrix}
        \epsilon \\ 1
    \end{bmatrix} 1^{-1} \begin{bmatrix}
        1 & 0
    \end{bmatrix} = \begin{bmatrix}
        \epsilon & 0 \\
        1 & 0
    \end{bmatrix},
\end{equation*} which is a reasonable approximation as $\norm{A-A_{2,1}}_F = 1 \approx \sigma_{2}(A) \approx 1-\epsilon/2$. 

The significance of controlling the $\norm{Q_C(I,:)^{-1}}_2$ term will also be highlighted when we analyze the numerical stability of the CUR decomposition in the subsequent section. 

\subsection{Numerical Stability of the CUR decomposition} \label{subsec:numstab}
In the absence of rounding errors, the error for the CURCA can be bounded by Theorem \ref{thm:CAe} and Corollary \ref{cor:CAe}. In this section, we derive an error bound that accounts for rounding errors. 

We use the standard model of floating-point arithmetic as in \cite[Section 2.2]{HighamStabilityBook}:
\begin{equation*}
    fl(x \,\, \mathrm{op} \,\, y) = (x \,\, \mathrm{op}\, \, y)(1+\delta), \, |\delta| \leq u
\end{equation*} where $\mathrm{op} \in \{+,-,*,/\}$ is the basic arithmetic operations and $u \ll 1$, the unit roundoff, is the precision at which the computations are performed. We use $fl(\cdot)$ and $\widehat{\cdot}$ to denote the computed value of the expression. We define
$\gamma := p(m,n,k) u$ where $p(m,n,k)$ is a low-degree polynomial in $m,n$ and $k$, and use $\gamma$ to suppress any constant factors and terms related to the size of the matrix and the target rank, e.g., $\sqrt{m},\sqrt{n}$ and $k$, but not $\sigma_i(A)$ or $1/\epsilon$. While this may appear as an oversimplication, this approach is standard practice in stability analysis; see e.g. \cite{Nakatsukasa2020,NakatsukasaHigham2012}. 

In this section, we denote the $i$th row of a matrix $B$ as $[B]_i$ and we use $C = A(:,J)$, $U = A(I_*,J)$ and $R = A(I,:)$ for shorthand. We assume that the rows are oversampled, i.e., $I_* = I\cup I_0$ is such that the number of the oversampling indices $I_0$ is bounded by a constant times $k$ where $k = |I| = |J|$ is the target rank, and the truncation parameter $\epsilon$ satisfies $\norm{A}_2 \gg \epsilon > \gamma\norm{A}_2$. We begin by stating two lemmas that will be used in the CURCA stability analysis. Lemma \ref{lemma:NB1} proves the perturbation bound for the projector $CU_\epsilon^\dagger \Pi_J$ when only $C$ and $U$ get perturbed and in Lemma \ref{lemma:NB2}, we prove that under perturbation on $C$ and $U$, $\tilde{C}\tilde{U}_\epsilon^\dagger \Pi_J$ approximately projects $C$ onto itself. The proofs for the two lemmas can be found in Appendix \ref{app:lemmas}.
\begin{lemma} \label{lemma:NB1}
    Under Assumption \ref{assumptions}, for any $\Delta C$ and $\Delta U$,
    \begin{equation}
        \norm{(C+\Delta C)(U+\Delta U)_\epsilon^\dagger}_2 \leq \norm{Q_C(I_*,:)^\dagger}_2 \left(1 + \frac{1}{\epsilon}\norm{\Delta U}_2 \right) + \frac{1}{\epsilon}\norm{\Delta C}_2.
    \end{equation}
\end{lemma}
\begin{lemma} \label{lemma:NB2}
    Under Assumption \ref{assumptions}, for any $\Delta C$ and $\Delta U$,
    \begin{equation}
        (C+\Delta C)(U+\Delta U)_\epsilon^\dagger R\Pi_J = C + E_*
    \end{equation} where 
    \begin{align*}
         \norm{E_*}_2 \leq \norm{(Q_C(I_*,:)^\dagger}_2 \left(\epsilon  + 2\norm{\Delta U}_2 + \frac{1}{\epsilon}\norm{\Delta U}_2^2\right) +  \norm{\Delta C}_2 \left(1+\frac{\norm{\Delta U}_2}{\epsilon}\right).
     \end{align*}
\end{lemma}

We now begin with the stability analysis. Note that the stability depends on the specific implementation used. 
In the forthcoming analysis, we assume the SCURCA $A_{I_*J}^\epsilon$ is computed as follows.\footnote{Note that the stability crucially depends on the implementation. Other implementations are possible, an obvious one being one that computes $C(U^\dagger R)$ rather than $(CU^\dagger )R$ as done here. This is seen to work well too, although we do not have a proof.}
\begin{enumerate}
    \item Compute the factors $\hat{C} = fl(A\Pi_J), \hat{U} = fl(\Pi_{I_*}^TA\Pi_J)$ and $\hat{R} = fl(\Pi_{I_*}^TA)$.
    \item Solve the (rank-deficient) underdetermined linear systems, 
    \begin{equation*}
        \hat{s}_i^{(1)} = fl\left((\hat{U}^T)_\epsilon^\dagger [\hat{C}]_i^T\right) \in \R^{k+p}
    \end{equation*}
    for all $i\in [m]$, where $[\hat{C}]_i$ is the $i$th row of $\hat{C}$.
    \item Compute matrix-vector multiply $\hat{s}_i^{(2)} = fl\left(\hat{R}^T\hat{s}_i^{(1)}\right) \in \R^n$.
    \item Let $\left(\hat{s}_i^{(2)}\right)^T$ be the $i$th row of the computed CUR decomposition $fl\left(A_{I_*J}^\epsilon\right)$.
\end{enumerate}

We now analyze each step. The first step is matrix-matrix multiplications with orthonormal matrices. Using the forward error bound\footnote{The error is termed the forward error as it indicates how close the computed version $\hat{C}$ is to the exact version $C$ \cite[Section 1.5]{HighamStabilityBook}.} for matrix-matrix multiplication \cite[Section 3.5]{HighamStabilityBook}, we obtain the following:
\begin{itemize}
    \item $\hat{C} = C + E_C$ where $\norm{E_C}_2 \leq \gamma \norm{A}_2$,
    \item $\hat{U} = U + E_U$ where $\norm{E_U}_2 \leq \gamma \norm{A}_2$,
    \item $\hat{R} = R + E_R$ where $\norm{E_R}_2 \leq \gamma \norm{A}_2$.
\end{itemize} In many cases, these error matrices $E_C,E_U$ and $E_R$ are the zeros matrix as $C, U$ and $R$ are simply submatrices of the original matrix $A$.

In the second step, we solve the (rank-deficient) underdetermined linear systems row by row. The error analysis for the (rank-deficient) underdetermined linear systems can be summarized in the following theorem. The proof of Theorem \ref{thm:rankdef} can be found in Appendix \ref{app:stab}.
\begin{theorem} \label{thm:rankdef}
    Consider the (rank-deficient) underdetermined linear system,
    \begin{equation} \label{eq:rankdefeq}
        \min_{x'} \norm{B_\epsilon x'-b}_2
    \end{equation} where $B_\epsilon \in \R^{m\times n}$ $(m\leq n)$
is (possibly) rank-deficient $(\rk(B_\epsilon)\leq m)$ with singular vales larger than $\epsilon$ and $b\in \R^{m}$. Then assuming $\epsilon > \gamma\norm{B_\epsilon}_2$, the minimum norm solution to \eqref{eq:rankdefeq} can be computed in a backward stable manner, i.e., the computed solution $\hat{s}$ satisfies
    \begin{equation}
        \hat{s} = \left( B_\epsilon + E_1 \right)^\dagger \left(b+E_2\right)
    \end{equation} where $\norm{E_1}_2 \leq \gamma \norm{B_\epsilon}_2$ and $\norm{E_2}_2\leq \gamma\norm{b}_2$.
\end{theorem}

Theorem \ref{thm:rankdef} tells us that the computed solution to a (rank-deficient) underdetermined linear system is the exact solution to a slightly perturbed problem. Now, using Theorem \ref{thm:rankdef}, for each $i\in [m]$ we obtain 
\begin{equation} \label{eq:rankdeflssol}
    \hat{s}_i^{(1)} = \left((\hat{U}^T)_\epsilon + E_i^{(U)} \right)^\dagger \left([\hat{C}]_i^T+ E_i^{(C)}\right),
\end{equation} where $\norm{E_i^{(U)}}_2 \leq \gamma \norm{A}_2$ and $\norm{E_i^{(C)}}_2\leq \gamma \norm{A}_2$. It is worth emphasizing that the backward errors $E_i^{(C)}$, $E_i^{(U)}$ depend on $i$.
Now since $\epsilon > \gamma\norm{A}_2$ by assumption and $\sigma_{\min}\left(\hat{U}^T_\epsilon + E_i^{(U)}\right)\geq \epsilon-\gamma\norm{A}_2$ by Weyl's inequality, there exists a perturbation $E_i\in \R^{k\times (k+p)}$ with $\norm{E_i}_2 \leq \epsilon + \gamma \norm{A}_2$\footnote{If $\hat{U}^T = W_1\Sigma_1 V_1^T + W_2 \Sigma_2 V_2^T$ is the SVD of $\hat{U}^T$ where $\Sigma_2$ contains the singular values of $\hat{U}^T$ smaller than $\epsilon$, then we can take $E_i = -W_2\Sigma_2V_2^T+E_i^{(U)}$ for each $i$.} such that 
\begin{equation}
    \hat{s}_i^{(1)} = \left(\hat{U}^T + E_i \right)_{\epsilon-\gamma\norm{A}_2}^\dagger \left([\hat{C}]_i^T+ E_i^{(C)}\right).
\end{equation} For shorthand, let $\hat{S}$ be a matrix with its $i$th row equal to $\left(\hat{s}_i^{(1)}\right)^T$.

In the third step, we compute a matrix-vector product \cite[Section 3.5]{HighamStabilityBook} with $\hat{R}^T$, which gives us $\hat{s}_i^{(2)} = fl\left(\hat{R}^T\hat{s}_i^{(1)}\right) = \hat{R}^T\hat{s}_i^{(1)} + E_{s_i}$ with
\begin{align} \label{eq:hatSbound}
    \norm{E_{s_i}}_2 &\leq \gamma \norm{\hat{R}^T}_2\norm{\hat{s}_i^{(1)}}_2 \notag \\ &\leq \gamma\norm{A}_2\left(\norm{Q_C(I_*,:)^\dagger}_2 \left(1+\frac{\norm{E_i}_2+\norm{E_U}_2}{\epsilon-\gamma\norm{A}_2}\right)+\frac{\norm{E_i^{(C)}}_2+\norm{E_C}_2}{\epsilon-\gamma\norm{A}_2} \right) \notag \\
    &\leq \gamma\norm{A}_2\left(\norm{Q_C(I_*,:)^\dagger}_2 \left(1+\frac{(\epsilon+\gamma \norm{A}_2)+\gamma\norm{A}_2}{\epsilon-\gamma\norm{A}_2}\right)+\frac{\gamma\norm{A}_2+\gamma\norm{A}_2}{\epsilon-\gamma\norm{A}_2} \right) \notag \\
    &\leq \gamma \norm{A}_2 \norm{Q_C(I_*,:)^\dagger}_2 ,
\end{align} where Lemma \ref{lemma:NB1} was used in the penultimate line with $\Delta U = E_U+ E_i^T$ and $\Delta C = E_C + e_i E_i^{(C)}$ where $e_i\in \R^{m}$ is the $i$th canonical basis vector. In the last line of \eqref{eq:hatSbound}, we used the fact that $\gamma$ suppresses any low-degree polynomial in $m,n$ and $k$, and $\epsilon > \gamma\norm{A}_2$. 

The following shows the expression for $\hat{s}_i^{(2)}$,
\begin{align*}
    \hat{s}_i^{(2)} &= \hat{R}^T\hat{s}_i^{(1)} + E_{s_i} = \hat{R}^T\left(\hat{U}^T + E_i \right)_{\epsilon-\gamma\norm{A}_2}^\dagger \left([\hat{C}]_i^T+ E_i^{(C)}\right)+ E_{s_i} \\
    &= (R+E_R)^T\left(U^T + E_U^T + E_i\right)_{\epsilon-\gamma\norm{A}_2}^\dagger \left([C+E_C]_i^T + E_i^{(C)}\right)  + E_{s_i}.
\end{align*}

Finally, in the fourth step, we combine $\hat{s}_i^{(2)} \in \R^n$ into a matrix to form $\widehat{A_{I_*J}^\epsilon} = fl\left(A_{I_*J}^\epsilon\right) \in \R^{m\times n}$ by setting the $i$th row of $\widehat{A_{I_*J}^\epsilon}$ to be  $\left(\hat{s}_i^{(2)}\right)^T$, giving us
\begin{equation} \label{eq:fleqn}
    fl\left(A_{I_*J}^\epsilon\right) = \hat{S}\hat{R} + E
\end{equation} where $E\in \R^{m\times n}$ is a matrix with its $i$th row equal to $E_{s_i}^T$ and satisfies $\norm{E}_2 \leq \gamma \norm{A}_2 \norm{Q_C(I_*,:)^\dagger}_2$, since $\gamma$ suppresses any low-degree polynomial in $m$. 

We now state the main stability result of the CURCA with the $\epsilon$-pseudoinverse, $A_{I_*J}^\epsilon$.
\begin{theorem} \label{thm:stabresult}
Let $0<\epsilon \ll 1$ be a truncation parameter for the pseudoinverse such that $\epsilon > \gamma \norm{A}_2$. Suppose that $\sloppy A_{I_*J}^\epsilon = A(:,J)A(I_*,J)_\epsilon^\dagger A(I_*,:)$ is computed in the following order:
    \begin{enumerate}
        \item Compute $C = A(:,J)$, $U = A(I_*,J)$ and $R = A(I_*,:)$,
        \item Compute each row of $CU_\epsilon^\dagger$ using a backward stable (rank-deficient) \sloppy underdetermined linear solver,
        \item Compute $CU_\epsilon^\dagger$ times $R$. Let $fl\left(A_{I_*J}^\epsilon\right)$ denote the output.
    \end{enumerate}  Then under Assumption \ref{assumptions},
    \begin{align} 
        \norm{A - fl\left(A_{I_*J}^\epsilon\right)}_F &\leq 4\sqrt{m} \norm{Q_C(I_*,:)^\dagger}_2 \norm{Q_X(J,:)^\dagger}_2\left(\norm{A(I-X^\dagger X)}_F + 2\epsilon\right) \\
    &\quad + \gamma \norm{A}_2 \norm{Q_C(I_*,:)^\dagger}_2. \notag
    \end{align}
\end{theorem}

Theorem \ref{thm:stabresult} tells us that the bound for the computed version of the stabilized CURCA $\widehat{A_{IJ}^\epsilon}$ is at most a factor $\bigO(\sqrt{m})$ worse than its exact arithmetic counterpart $A_{IJ}^\epsilon$ plus an error of $\gamma \norm{A}_2 \norm{Q_C(I_*,:)^\dagger}_2$. More specifically, we have
\begin{equation*}
    \norm{A-fl(A_{I_*J}^\epsilon)}_F \leq 4\sqrt{m} \left(\text{Bound \eqref{eq:OSThm}}\right) +\gamma \norm{A}_2 \norm{Q_C(I_*,:)^\dagger}_2.
\end{equation*}
Roughly, this shows that the computed error is in the same order as the error in exact arithmetic, up to a factor $\bigO(\sqrt{m})$.
The factor $\bigO(\sqrt{m})$ is likely an artifact of the proof given below; in stability analysis it is common to see bounds with such overestimates, and also standard practice to expect to observe much better performance in practice. 
Throughout various parts of the proof, we loosely bound the $2$-norm by the Frobenius norm, typically because we only have information about the norms of rows or columns of a matrix. For example in the proof of Theorem \ref{thm:stabresult}, we bound $\norm{\hat{S}}_2 \leq \norm{\hat{S}}_F$ in \eqref{eq:hatSineq}, which is likely a pessimistic overestimate.

\begin{proof}[Proof of Theorem \ref{thm:stabresult}]
    From the above analysis, we have
    \begin{equation} \tag{\ref{eq:fleqn}}
         fl\left(A_{I_*J}^\epsilon\right) = \hat{S}\hat{R} + E = \hat{S}R+\hat{S}E_R + E,
    \end{equation} where $\norm{E}_2 \leq \gamma \norm{A}_2 \norm{Q_C(I_*,:)^\dagger}_2$. Let us apply Lemma \ref{lemma:NB2} to each row of $\hat{S}R$ to obtain
    \begin{equation*}
        \hat{S}_i R\Pi_J = \left([C+E_C]_i +\left(E_i^{(C)}\right)^T\right)(U+E_U+E_i^T)_{\epsilon-\gamma \norm{A}_2}^\dagger R \Pi_J = [C]_i + E_i^{(P)},  
    \end{equation*} where $\norm{E_i^{(P)}}_2 \leq 8\epsilon \norm{Q_C(I_*,:)^\dagger}_2$. Therefore, letting $E^{(P)}$ be a matrix with $E_i^{(P)}$ as its $i$th row, we get
    \begin{equation*}
        \hat{S}R\Pi_J = C + E^{(P)},
    \end{equation*} where $\norm{E^{(P)}}_F \leq 8\epsilon\sqrt{m} \norm{Q_C(I_*,:)^\dagger}_2$. Now proceeding similarly to the proof of Theorem \ref{thm:CAe}, we obtain
\begin{align} \label{eq:proof1}
     A -  \hat{S}R &= \left(I- \hat{S}\Pi_{I_*}^T\right)A = \left(I- \hat{S}\Pi_{I_*}^T\right)A(I-\Pi_J (X\Pi_J)^\dagger X) - E^{(P)}(X\Pi_J)^\dagger X \notag\\
     &= \left(I- \hat{S}\Pi_{I_*}^T\right)A(I-X^\dagger X)(I-\Pi_J (X\Pi_J)^\dagger X) - E^{(P)}(X\Pi_J)^\dagger X,
\end{align} where $E^{(P)}(X\Pi_J)^\dagger X$ satisfies
\begin{equation*}
    \norm{E^{(P)}(X\Pi_J)^\dagger X}_F \leq \norm{E^{(P)}}_F \norm{(X\Pi_J)^\dagger X}_2 \leq  8\epsilon\sqrt{m} \norm{Q_C(I_*,:)^\dagger}_2 \norm{Q_X(J,:)^\dagger}_2.
\end{equation*}
The first term in \eqref{eq:proof1} can be bound by
\begin{align*}
     \norm{I- \hat{S}\Pi_{I_*}^T}_2 & \norm{A(I-X^\dagger X)}_F \norm{I-\Pi_J (X\Pi_J)^\dagger X}_2 \\ &\leq \left(1+\norm{\hat{S}}_2\right) \norm{A(I-X^\dagger X)}_F \norm{\Pi_J (X\Pi_J)^\dagger X}_2 \\
     &\leq 4\sqrt{m}\norm{Q_C(I_*,:)^\dagger}_2 \norm{Q_X(J,:)^\dagger}_2\norm{A(I-X^\dagger X)}_F
\end{align*} where in the final line we used \eqref{eq:hatSbound}, noting that
\begin{equation} \label{eq:hatSineq}
    \norm{\hat{S}}_2 \leq \norm{\hat{S}}_F \leq \sqrt{\sum_{i=1}^m \norm{\hat{s}_i^{(1)}}_2^2}  \leq 3\sqrt{m}\norm{Q_C(I_*,:)^\dagger}_2.
\end{equation} Using \eqref{eq:hatSineq}, we can also bound
\begin{equation*}
    \norm{\hat{S}E_R}_F \leq \norm{\hat{S}}_F\norm{E_R}_2 \leq \gamma \norm{A}_2 \norm{Q_C(I_*,:)^\dagger}_2.
\end{equation*}
Finally putting everything together, we obtain
\begin{align*}
    \norm{A-fl(A_{I_*J}^\epsilon)}_F &= \norm{A-\hat{S}R -\hat{S}E_R-E}_F \\ &\leq 4\sqrt{m} \norm{Q_C(I_*,:)^\dagger}_2 \norm{Q_X(J,:)^\dagger}_2\left(\norm{A(I-X^\dagger X)}_F + 2\epsilon\right) \\
    &\quad + \gamma \norm{A}_2 \norm{Q_C(I_*,:)^\dagger}_2.
\end{align*}
\end{proof}

It is natural to wonder what could go wrong without the $\epsilon$-pseudoinverse. Two problems may arise without the $\epsilon$-pseudoinverse. First, the matrix $U$ may not be numerically full rank, so Theorem \ref{thm:rankdef} cannot be used as $\epsilon>\gamma\norm{A}_2$ may no longer hold. In addition, the two lemmas, Lemma \ref{lemma:NB1} and Lemma \ref{lemma:NB2} become meaningless as both require division by $\epsilon$. Without the $\epsilon$-pseudoinverse, the error bound can become uncontrollably large, causing issues in several places in the proof of Theorem \ref{thm:stabresult}. For example, the bound for the error $\norm{E_{s_i}}_2$ in \eqref{eq:hatSbound} for computing the matrix-vector product and the bound for $\norm{E_i^{(P)}}_2$ in the proof of Theorem \ref{thm:stabresult} may no longer hold as they can become arbitrarily large without the $\epsilon$-pseudoinverse. Nevertheless, in practice we observe stability without the $\epsilon$-truncation, so we recommend a careful implementation of the pseudoinverse (without the $\epsilon$-truncation) for practical purposes.\footnote{To be clear, implementing the $\epsilon$-truncation does not increase the complexity and can be recommended for guaranteed stability. We have simply not observed instability without the $\epsilon$-truncation.}
A similar observation, commenting on the role of the $\epsilon$-pseudoinverse, has been mentioned in \cite[Section 4.2]{Nakatsukasa2020}. See Section \ref{subsec:implementation} for numerical experiments.

In Theorem \ref{thm:stabresult}, one might question how large $\norm{Q_C(I_*,:)^\dagger}_2$ can be, given that it is part of the added term, $\gamma \norm{A}_2 \norm{Q_C(I_*,:)^\dagger}_2$ in Theorem \ref{thm:stabresult}. This could pose a problem if $\norm{Q_C(I_*,:)^\dagger}_2$ grows exponentially in $m$ or $n$. However, Theorem \ref{thm:stabresult} is enough to conclude that $A_{IJ}^\epsilon$ is numerically stable when the indices are chosen reasonably. This means that we first choose the column indices $J$, and sensibly select the row indices $I_*$ from the the selected columns $A(:,J)$ such that $\norm{A(:,J)A(I_*,J)^\dagger}_2 = \norm{Q_C(I_*,:)^\dagger}_2$ is bounded by a low-degree polynomial involving $m,n$ and $k$; see Section \ref{subsec:chooseRCdep}. A similar approach is also employed in other works such as \cite{CortinovisYing2024,DongMartinsson2023,VoroninMartinsson2017,Xia2024}. For example, Gu-Eisenstat's strong rank-revealing QR factorization \cite{GuEisenstat1996} can be used on the chosen columns $A(:,J)$ to obtain $\norm{Q_C(I,:)^{-1}}_2 \leq \sqrt{mk}$, which can be reduced further with oversampling. This also highlights the importance of oversampling. In the absence of oversampling, poorly selected indices can make $\norm{Q_C(I,:)^{-1}}_2$ exponentially large. Therefore, oversampling can be employed to stabilize the CURCA, which we discuss further in the following section. 

\section{Oversampling for the CURCA} \label{sec:method}
In the previous section, we proved theoretical results involving the CURCA (Corollary \ref{cor:CAe}), the stabilized CURCA (Theorem \ref{thm:CAe}) and the stabilized CURCA in the presence of rounding errors (Theorem \ref{thm:stabresult}). All of the results involved an oversampling parameter $p$ and an extra set of row indices $I_0$ and bounding $\norm{Q_C(I\cup I_0,:)^{\dagger}}_2$ was important for the accuracy and stability of the CURCA. In this section, we discuss oversampling in the context of the CURCA and devise an algorithm that naturally arises from the discussion. 

We first describe the setting. Suppose we have obtained the row indices $I$ and the column indices $J$ with $|I| = |J| = k$ by applying some algorithm, for example, the ones discussed in the introduction (Section \ref{sec:intro}), on a row space approximator $X \in \R^{k\times n}$ of $A\in \R^{m\times n}$.\footnote{A similar version for column space approximator can also be devised by considering $A^T$ instead. In the case when $X$ and $Y$ are not available, for example, when the initial set of indices were obtained using uniform sampling, $R  = A(I,:)$ and $C = A(:,J)$ can be used as the row space approximator and the column space approximator, respectively.}  To have a concrete algorithm in mind for getting the set of indices $I$ and $J$, we present pivoting on a random sketch \cite{DongMartinsson2023,Duersch_Gu_SIREV,VoroninMartinsson2017} below in Algorithm \ref{alg:pivsketch}.

\begin{algorithm}[!h]\footnotesize
  \caption{Pivoting on a random sketch (\cite[Algorithm 1]{DongMartinsson2023})}
  \label{alg:pivsketch}
  \begin{algorithmic}[1]
\Require{$A \in \R^{m\times n}$ of rank $r$, target rank $k \leq r$ (typically $k \ll \min\{m,n\}$)}
\Ensure{Column indices $J$ and row indices $I$ with $|I| = |J| = k$}
\vspace{0.5pc}
\Statex \texttt{function} $[I,J] = \mathtt{Rand\_ Pivot}(A,k)$
\State Draw a random embedding $\Omega\in \R^{k\times m}$.
\State Set $X = \Omega A \in \R^{k\times n}$, a row sketch of $A$
\State Apply CPQR on $X$. Let $J$ be the $k$ column pivots.
\State Apply CPQR on $A(:,J)^T$. Let $I$ be the $k$ row pivots.
\end{algorithmic}
\end{algorithm}

Algorithm \ref{alg:pivsketch} is a version of Algorithm 1 from \cite{DongMartinsson2023}, which selects the column indices first by applying column pivoted QR (CPQR)\footnote{LU with partial pivoting is also effective in practice \cite{DongMartinsson2023}.} on the row sketch $X = \Omega A$ and then selects the row indices by applying CPQR on the chosen columns $A(:,J)^T$. Algorithm \ref{alg:pivsketch} is an example where we obtain the row indices from the already-chosen column indices, which was recommended in the previous section. Here, the row space approximator is the row sketch $X = \Omega A$ and the column space approximator is the columns $A(:,J)$. A bound for the CURCA (see Corollary \ref{cor:CAe}) without oversampling is given by
\begin{equation} \label{eq:CAbnd}
    \norm{A-A_{IJ}}_F \leq \norm{Q_C(I,:)^{-1}}_2 \norm{Q_X(J,:)^{-1}}_2  \norm{A-AX^\dagger X}_F
\end{equation} where $A_{IJ} = A(:,J)A(I,J)^\dagger A(I,:)$ is the CURCA and $Q_C$ and $Q_X$ are orthonormal matrices spanning the columns of $C$ and $X^T$ respectively. 

Many existing algorithms focus on minimizing the first two terms on the right-hand side of \eqref{eq:CAbnd} as they control the accuracy of the CURCA. In the CURCA and the other CUR decompositions such as the CURBA, $C(C^\dagger A R^\dagger )R$, we often take $|I| = |J|$, however, without increasing the overall rank of the approximation, we can oversample either $I$ or $J$.\footnote{Oversampling both the row indices $I$ and the column indices $J$ independently is not recommended. See Section \ref{subsec:oversampleNI} for a further discussion.} Suppose we oversample the rows to $I_*:= I\cup I_0$ where $|I_0| = p$ is an extra set of row indices for oversampling. Then the first term of the bound changes from $\norm{Q_C(I,:)^{-1}}_2$ to $\norm{Q_C(I_*,:)^{\dagger}}_2$ (see Corollary \ref{thm:CAe}). Now $\sigma_{\min}(Q_C(I_*,:))\geq \sigma_{\min}(Q_C(I,:))$ by the Courant-Fischer min-max theorem, which improves the bound in \eqref{eq:CAbnd} as $\norm{Q_C(I_*,:)^{\dagger}}_2 \leq \norm{Q_C(I,:)^{-1}}_2$. 
Now, to maximize the effect of oversampling, we ought to find unchosen indices that enrich the trailing singular subspace of $Q_C(I,:)$, which in turn increases the minimum singular value(s) of $Q_C(I,:)$. It turns out that we can achieve this by projecting $Q_C([m]-I,:)$ onto the trailing singular subspace of $Q_C(I,:)$ and use a good row selection algorithm to choose $p$ extra rows. Before stating the algorithm, we first motivate our rationale behind our approach using the cosine-sine (CS) decomposition.

Let $Q \in \R^{n\times k}$ be any orthonormal matrix with partition 
\begin{equation}
    Q = \begin{blockarray}{cc}
         k \\
       \begin{block}{[c] c}
         Q_1 & n_1 \\
         Q_2 & n_2 \\
       \end{block}
     \end{blockarray}
\end{equation}
where $n_1,n_2 \geq k$ with $n_1 + n_2 = n$. Then the CS decomposition \cite{PaigeWei1994} gives us
\begin{equation} \label{eq:CS2eqns}
    Q_1 = U_1 C V^T \text{ and } Q_2 = U_2 S V^T,
\end{equation} where $U_1 \in \R^{n_1\times k},U_2 \in \R^{n_2 \times k},V \in \R^{k\times k}$ are matrices with orthonormal columns and $C = \diag(c_1,c_2,...,c_k), S = \diag(s_1,s_2,...,s_k)$ are diagonal matrices satisfying $c_i^2 + s_i^2 = 1$ for all $i$. In our context, we can view $Q_C(I,:)$ as $Q_1$ and $Q_C([m]-I,:)$ as $Q_2$. Now assume, without loss of generality, that the $c_i$'s are in non-increasing order so the $s_i$'s are in non-decreasing order. Then in order to increase the minimum singular value of $Q_1$, we could add the rows of $Q_2$ that contribute the most to the trailing right singular subspace of $Q_1$, i.e., $V_{-p} = V(:,k-p+1:k) \in \R^{k\times p}$. By \eqref{eq:CS2eqns}, $V_{-p}$ is also the dominant right singular subspace of $Q_2$. When we add the rows of $Q_2$ that lies in the subspace spanned by $V_{-p}$ to $Q_1$, we can increase the minimum singular value(s) of $Q_1$, i.e., increase $c_k$ or the last few $c_i$'s. Therefore we can apply any algorithm (such as those discussed in Section \ref{sec:intro}) that finds good row indices on $Q_2V_{-p} \in \R^{n_2 \times p}$ and append them to $Q_1$ to increase the minimum singular value of $Q_1$. 

If the algorithm requires the dominant left singular vectors of $Q_2V_{-p}$ such as in DEIM or leverage scores sampling, we can simply scale the columns of $Q_2V_{-p}$ using the singular values of $Q_1$, $C = \diag(c_1,...,c_k)$ because
\begin{equation*}
    Q_2V_{-p} = U_{2,-p}\diag(s_{k-p+1},...,s_k) = U_{2,-p}\sqrt{1-\diag(c_{k-p+1}^2,...,c_k^2)}
\end{equation*} where $U_{2,-p}  = U_2(:,k-p+1:k)$.

In light of the observation made above, we propose the following algorithm (Algorithm \ref{alg:oversample_parameter}) for oversampling indices. For simplicity, we use column pivoted QR (CPQR) on $Q_2V_{-p}$ in Algorithm \ref{alg:oversample_parameter} to obtain a good set of oversampling indices. However, any algorithm that finds a good set of rows can replace line $4$ of Algorithm \ref{alg:oversample_parameter}.

\begin{algorithm}[ht]\footnotesize
  \caption{Oversampling indices}
  \label{alg:oversample_parameter}
  \begin{algorithmic}[1]
\Require{A full column rank matrix $B\in \R^{n\times k}$ with $n\geq k$, an index set $I$ with $|I| = k$ and an oversampling parameter $p \leq k$}
\Ensure{Extra indices $I_0$ with $|I_0| = p$}
\vspace{0.5pc}
\Statex \texttt{function} $I_0 = \mathtt{OS}(B,I,p)$
\State $[Q_B,\sim ] = \mathtt{qr}(B,0)$, \Comment{Skip this step if $B$ is already orthonormal.}
\State $[\sim,\sim,V] = \mathtt{svd}(Q_B(I,:))$,
\State Set $V_{-p} = V(:,k-p+1:k)$, the trailing $p$ right singular vectors of $Q_B(J,:)$.
\State Apply CPQR on $\left(Q_B([m]-I,:)V_{-p}\right)^T$. Let $I_0$ be the extra $p$ indices for oversampling.
\end{algorithmic}
\end{algorithm}

Given a full column rank matrix $B$, an index set $I$ and an oversampling parameter $p (\leq k)$, Algorithm \ref{alg:oversample_parameter} finds the extra indices for oversampling by projecting $Q_B$ onto the unchosen indices $[m]-I$ from the left and the trailing $p$ right singular vectors of $Q_B(I,:)$ from the right and performing CPQR to obtain the extra indices $I_0$ with $|I_0| = p$. When $p >k$, we can iterate Algorithm \ref{alg:oversample_parameter} to obtain at most $k$ oversampling indices at each iteration. We can also devise a version of Algorithm \ref{alg:oversample_parameter} with a tolerance parameter $0<\epsilon<1$ rather than taking $p$ as input; for example $\epsilon=\sqrt{k/n}$ may be a reasonable choice (given that singular values of a $O(k)\times k$ submatrix of an $n\times n$ Haar distributed orthogonal matrix are of this order~\cite{meckes2019random}). This version uses projection onto the trailing right singular subspace of $Q_B(I,:)$ corresponding to the singular values that are less than $\epsilon$.
While this version does not guarantee $\sigma_{\min}\left(Q_B(I \cup I_0,:)\right)\geq \epsilon $, it can be applied iteratively to achieve this. To guarantee $\sigma_{\min}\left(Q_B(I \cup I_0,:)\right)\geq \epsilon$, one can alternatively use the \texttt{GappyPOD+E} oversampling algorithm~\cite{PeherstorferDrmacGugercin2020}. However, this approach comes with a higher computational cost.
The complexity of Algorithm \ref{alg:oversample_parameter} is $\bigO(nk^2)$ where the dominant cost comes from taking the QR decomposition of $B$ (line 1). If $B$ were orthonormal to begin with, then the dominant cost comes from forming $Q_B([m]-I,:)V_{-p}$, which costs $\bigO(nkp)$.

In the analysis for the CURCA in Section \ref{sec:analysis}, oversampling played two roles: (i) improve the accuracy of the CURCA (see Theorem \ref{thm:CAe}, Corollary \ref{cor:CAe}, Theorem \ref{thm:stabresult}), and (ii) improve the stability of the CURCA in the presence of roundoff errors (see Theorem \ref{thm:stabresult}). Therefore, it is important to oversample, especially if the original set of indices are not good. For example, if the core matrix $A(I,J)$ is (nearly) singular. Oversampling should be done in such a way that the term $\norm{Q_C(I,:)^{-1}}_2$ is reduced further by adding an extra set of indices $I_0$ that lie in the trailing singular subspace of $Q_C(I,:)$. Algorithm \ref{alg:oversample_parameter} achieves this by picking good unchosen indices from the trailing singular subspace of $Q_C$. We further highlight the significance of oversampling through numerical illustrations in Section \ref{sec:numill}.


\section{Numerical Illustration} \label{sec:numill}
In this section, we illustrate the concepts discussed in the previous sections through numerical experiments. We first discuss implementation details of (S)CURCA in MATLAB. Then we show that, without loss of generality, after selecting the rows first, the columns should be chosen based on those rows, i.e., the rows and the columns for the CURCA should not be chosen independently. We show that oversampling can improve the quality when the indices are poorly selected. Additionally, we also illustrate the effectiveness of the oversampling algorithm, Algorithm \ref{alg:oversample_parameter}, for the CURCA and show its competitiveness against some existing methods. In all the experiments, the best rank-$k$ approximation error using the truncated SVD (TSVD) is used as reference. The experiments were conducted in MATLAB version 2021a using double precision arithmetic.

\subsection{Implementation of (S)CURCA} \label{subsec:implementation}
The main concern in the computation of the CURCA is that (i) the representation of the CURCA typically involves three (highly) ill-conditioned matrix and (ii) we take the pseudoinverse of a (highly) ill-conditioned matrix. When the original matrix $A$ is low-rank and we have a good low-rank approximation of $A$, then we expect $C,U$ and $R$ to be highly ill-conditioned. We test four possible implementations in MATLAB,
\begin{enumerate}
    \item $\mathtt{A_{IJ}^{(1)} = \left(A(:,J)/A(I,J)\right)*A(I,:)}$,
    \item $\mathtt{A_{IJ}^{(2)} = A(:,J)*\left(V/S*W'\right)*A(I,:)}$ where $\mathtt{[W,S,V] = svd(U,`econ')}$,
    \item $\mathtt{A_{IJ}^{(3)} = \left(A(:,J)*V/S\right)*\left(W'*A(I,:)\right)}$ where $\mathtt{[W,S,V] = svd(U,`econ')}$,
    \item $\mathtt{A_{IJ}^{(\epsilon)} = \left(A(:,J)*V(:,II)/S(II,II)\right)*\left(W(:,II)'*A(I,:)\right)}$ \\ where $\mathtt{[W,S,V] = svd(U,`econ')}$ and $\mathtt{II = (diag(S) > \epsilon)}$ with $\epsilon = 10^{-15}$.
\end{enumerate} 
The third implementation, $\mathtt{A_{IJ}^{(3)}}$ is the suggested implementation of the CURCA. For guaranteed stability, we suggest the fourth implementation, $\mathtt{A_{IJ}^{(\epsilon)}}$ with the $\epsilon$-pseudoinverse. However we have not observed instability without the $\epsilon$-pseudoinverse using the third implementation. We use two test matrices: (1) $1000\times 1000$ test matrix generated using the MATLAB command, $\mathtt{randn(1000,30)*randn(30,1000)}$ and (2) $1374\times 1374$ matrix named $\mathtt{nnc1374}$ from the SuiteSparse Matrix Collection \cite{FloridaDataset2011}. We choose the rows and columns by first choosing the column indices $J$ using column pivoted QR (CPQR) on $A$ and then using CPQR on the chosen columns $A(:,J)$ to get the row indices $I$.

\begin{figure}[!ht]
\hspace*{-0.5cm}
\subfloat[\centering $\mathtt{randn(1000,30)*randn(30,1000)}$]{\label{subfig:implementation1}\includegraphics[scale = 0.33]{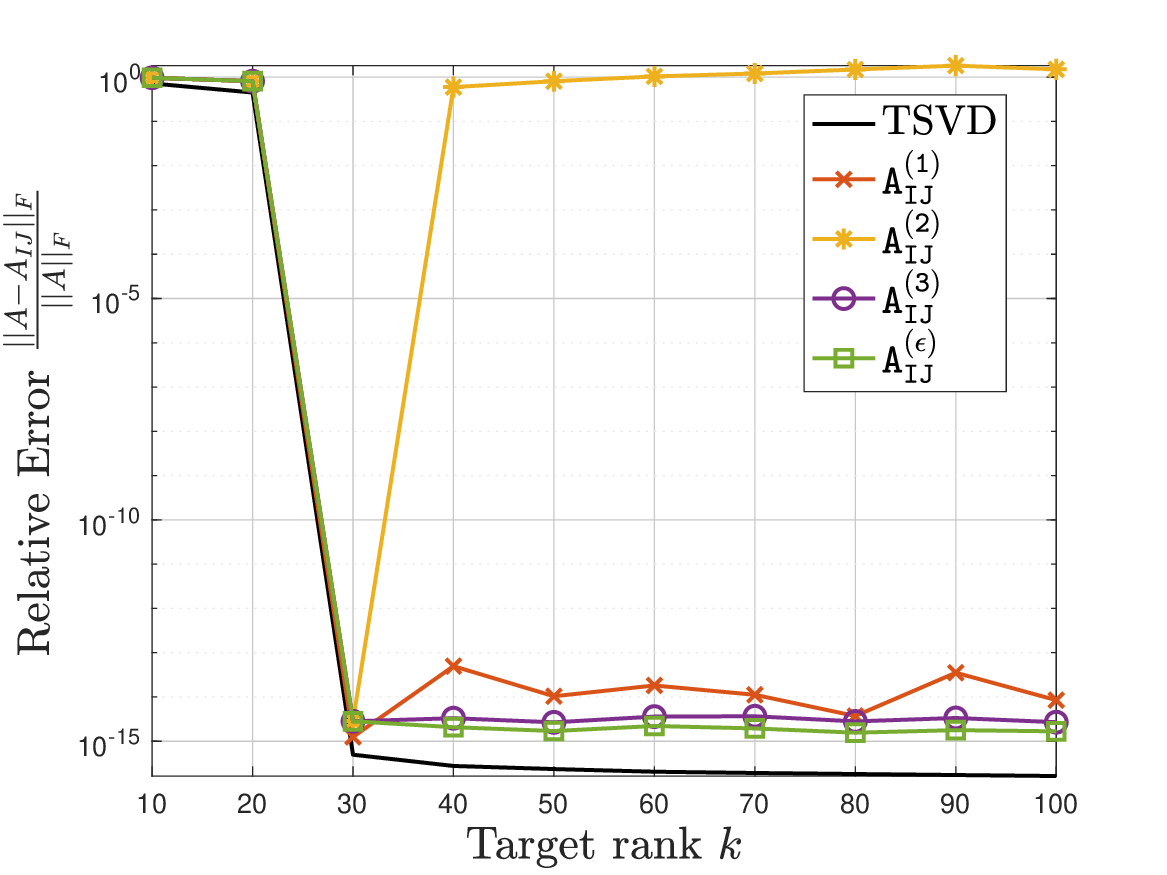}}
\subfloat[\centering $\mathtt{nnc1374}$]{\label{subfig:implementation2}\includegraphics[scale = 0.33]{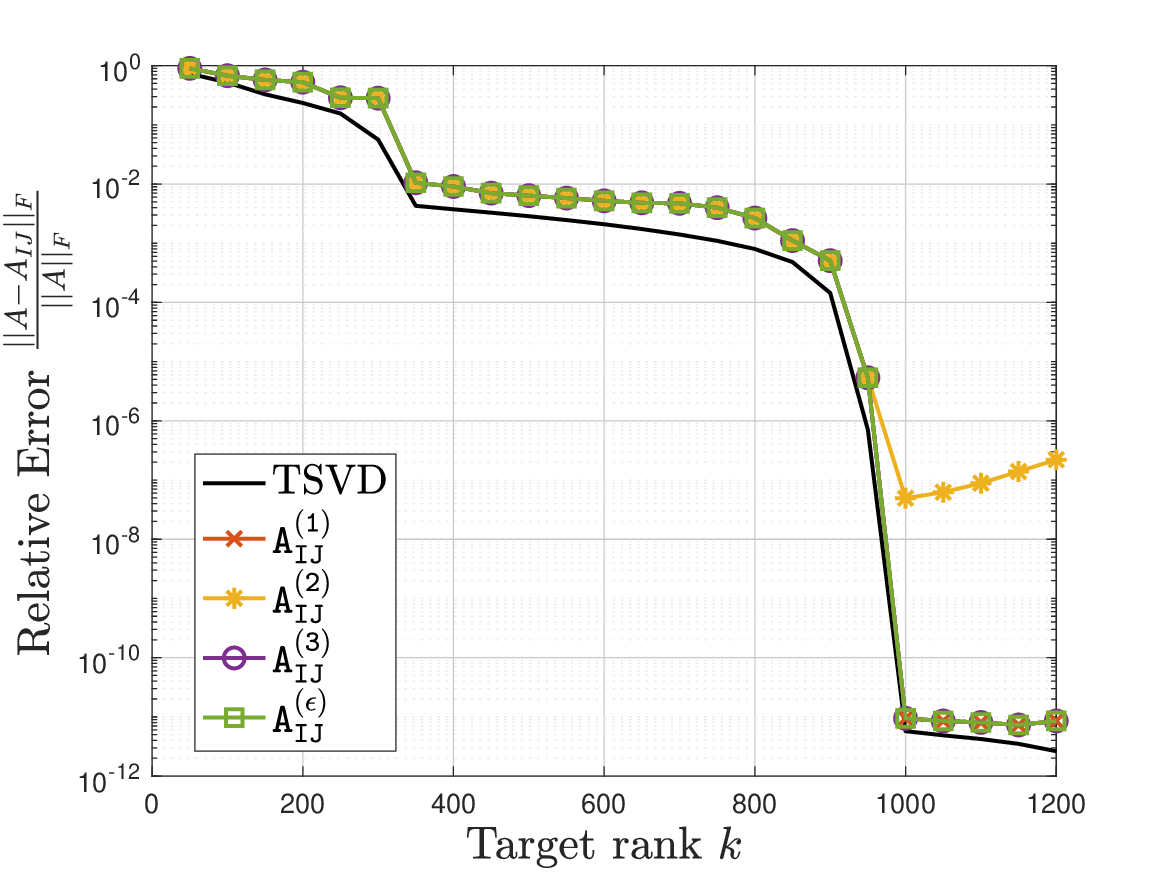}}
\centering 
\caption{Tests for different implementations of the CURCA. The third and fourth implementation performs stably. We recommend the third implementation.}
\label{fig:implementation}
\end{figure}

The results are depicted in Figure \ref{fig:implementation}. First, we notice that the third implementation, $\mathtt{A_{IJ}^{(3)}}$, which is the one we suggest, yields stable approximation throughout the experiment. In particular, the third implementation performs similarly to the fourth implementation with the $\epsilon$-pseudoinverse. On the other hand, the second implementation, $\mathtt{A_{IJ}^{(2)}}$ suffers from numerical errors as we lose accuracy when the singular values decay extremely rapidly. The reason is because we explicitly form the pseudoinverse of the core matrix $A(I,J)$, which is highly ill-conditioned, an observation also noted in \cite{Nakatsukasa2020}. Therefore, it is important to not form the pseudoinverse of $U = A(I,J)$ explicitly in the CURCA. Note also that the order in which the factors are multiplied is essential for numerical stability, as the second and third implementations only differ by the order in which the factors are computed. The first implementation, $\mathtt{A_{IJ}^{(1)}}$ may lose $1$ or $2$ orders of magnitude in Figure \ref{subfig:implementation1} once the target rank becomes larger than the rank of the test matrix and appear to be less stable than the third implementation. The slash commands $(\mathtt{/,\backslash })$ in MATLAB should be used with caution for underdetermined problems, as its backslash command applied to numerically rank-deficient underdetermined problems output a sparse solution based on a pivoting strategy~\cite[\S~2.4]{anderson1995lapack}, which can differ significantly from the minimum-norm solution and may not satisfy the assumptions in our analysis.
Note that using the $\mathtt{pinv}$ command with tolerance parameter $0$ for the CURCA in MATLAB, i.e., $\mathtt{A(:,J)*pinv(A(I,J),0),*A(I,:)}$, is equivalent to the second implementation, which can be unstable. For the rest of the numerical experiments, we use the third implementation
\begin{equation}
    \mathtt{\left(A(:,J)*V/S\right)*\left(W'*A(I,:)\right)} \text{ where } \mathtt{[W,S,V] = svd(U,`econ')}.
\end{equation}

The fourth implementation computes the stabilized CURCA and is implemented in such a way that the analysis in Section \ref{subsec:numstab} hold. A less expensive alternative is to use a rank-revealing QR factorization and truncate the bottom-right corner of the upper triangular factor and the relevant columns of the orthonormal factor corresponding to the diagonal elements less than $\epsilon$. In the rank-revealing QR, the diagonal elements give a good approximation to the singular values; see for example \cite{Chan1987,GuEisenstat1996}. A possible workaround without the $\epsilon$-pseudoinverse is also discussed in \cite{Nakatsukasa2020} where we perturb the core matrix $A(I,J)$ by a small noise matrix such that the singular values of $A(I,J)$ are all larger than the unit roundoff.

\subsection{Importance of not choosing rows and columns independently} \label{subsec:chooseRCdep}
In this section, we demonstrate the importance of choosing the rows and columns that are dependent on one another in the CURCA. If the rows and columns are chosen independently of each other, the matrix $U = A(I,J)$ can be (nearly) singular. In such a scenario, we show that a sufficient amount of oversampling can remedy this problem. We use two test matrices: (1) synthetic matrix generated using
\begin{equation} \label{eq:blkrandcountereg}
    A = \begin{bmatrix}
        10^{-10} \cdot \mathtt{randn(50,50)} & \mathtt{randn(50,950)} \\ \mathtt{randn(950,50)} & 0
    \end{bmatrix} \in \R^{1000\times 1000},
\end{equation} and (2) $\mathtt{nnc1374}$ matrix from the SuiteSparse Matrix Collection used in the previous section. We test the following four cases:
\begin{enumerate}
    \item Choose columns and rows independently by applying CPQR on $A$ and $A^T$ respectively,
    \item Choose columns and rows dependently by applying CPQR on $A$ and then applying CPQR on the chosen columns $A(:,J)$ to obtain the rows,
    \item Apply Case $1$ and additionally do
 row oversampling using Algorithm \ref{alg:oversample_parameter} with $p = k$,
    \item Apply Case $2$ and additionally do row oversampling using Algorithm \ref{alg:oversample_parameter} with $p = k$,
\end{enumerate} where $k$ is the target rank.

\begin{figure}[!ht]
\hspace*{-0.5cm}
\subfloat[\centering Block random matrix \eqref{eq:blkrandcountereg}]{\label{subfig:indep1}\includegraphics[scale = 0.33]{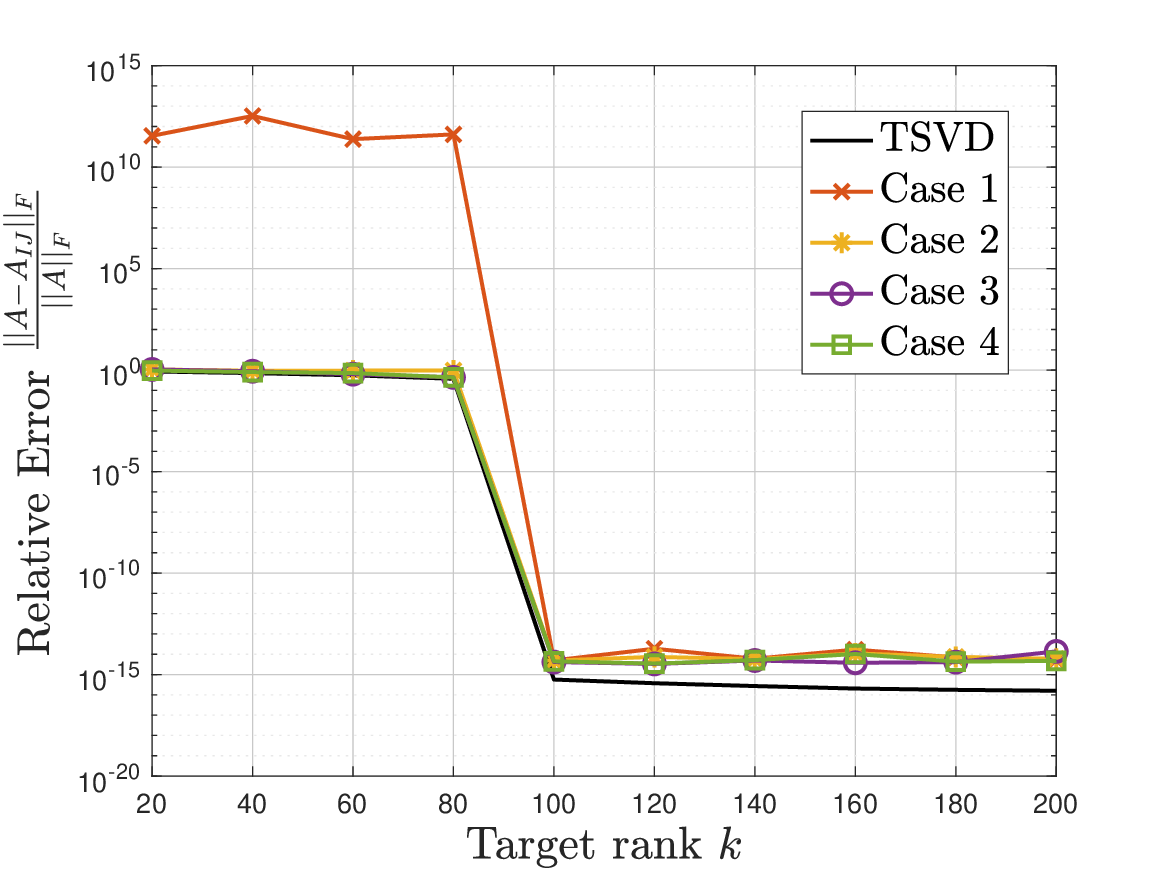}}
\subfloat[\centering $\mathtt{nnc1374}$]{\label{subfig:indep2}\includegraphics[scale = 0.33]{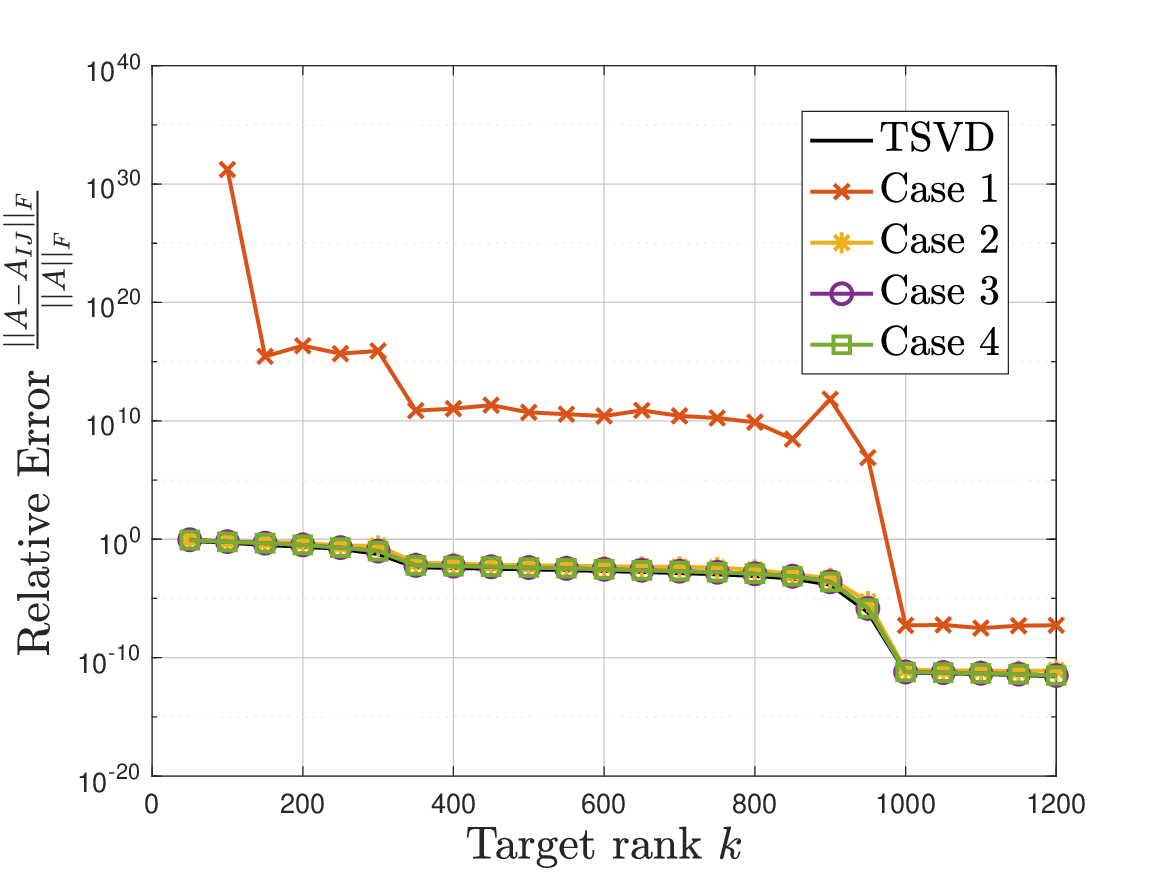}}
\centering 
\caption{Relationship between the rows and columns in the CURCA. In Case $1$, the rows and columns are chosen independently from one another and in Case $2$, we select the columns first and then the rows were computed from the selected columns. Cases $3$ and $4$ correspond to Cases $1$ and $2$ with row oversampling ($p = k$), respectively. When the rows and columns are chosen independently (Case $1$), the resulting approximation can be catastrophic.}
\label{fig:indep}
\end{figure}

In Figure \ref{fig:indep}, we show the importance of choosing the columns and rows in the CURCA. When the columns and rows are chosen independently (Case $1$), the resulting CURCA can be catastrophically poor. In Figure \ref{subfig:indep1}, poor approximation happens because when we choose the columns and rows independently, we choose the first 50 rows and columns, as they are the most important, implying that the core matrix is the small $(1,1)$-block of $A$. Since the chosen rows and columns are $\bigO(1)$ and their intersection is $\bigO(10^{-10})$, the CURCA becomes inaccurate.\footnote{A similar problem may also arise in random sampling where sampling rows and columns independently may lead to a poor CUR decomposition with high probability. For example, when 
$A = \big[
    \begin{smallmatrix}
    0 & 1 \\ 1 & 0      
    \end{smallmatrix}
\big]$
 and the target rank is $1$, the probability of sampling zero as the core matrix using leverage scores or column norms is $0.5$ when the rows and columns are sampled independently. Independent sampling leads to a poor CUR approximation $50\%$ of the time. However, if we sample them dependently, we choose $1$ as the core matrix with probability 1, leading to a sensible CUR approximation.} A similar issue also arises in Figure \ref{subfig:indep2} where the approximation becomes noticeably unreliable throughout; see also the $2\times 2$ example at the end of Section \ref{subsec:CURanal}. However, when we oversample sufficiently in Case $3$, we obtain a good CUR approximation. When we choose the rows and columns dependently in Cases $2$ and $4$, we have good stable approximations. Therefore, it is highly recommended to choose the rows and columns dependently. However, if that is not possible, a sufficient amount of oversampling is recommended.

\subsection{Oversampling algorithm comparison} \label{subsec:oversampleNI}
In this section, we illustrate advantages of oversampling through numerical experiments. In all experiments, unless stated otherwise, we use pivoting on a random sketch (Algorithm \ref{alg:pivsketch}) 
with column pivoted QR \cite{DongMartinsson2023,VoroninMartinsson2017} and the Gaussian sketch to get the initial set of $k$ row and column indices, where $k$ is the target rank. We then obtain the oversampling indices using the following algorithms:
\begin{enumerate}
    \item $\mathtt{OS+P}$: Algorithm \ref{alg:oversample_parameter},
    \item $\mathtt{OS+L}$: Choose $p$ extra indices corresponding to the largest $p$ leverage scores out of the unchosen indices as in \cite{GidisuHochstenbach2022},
    \item $\mathtt{OS+E}$: $\mathtt{GappyPOD+E}$ algorithm in \cite{PeherstorferDrmacGugercin2020}.
\end{enumerate} 
We set the number of oversampling indices to be $p=0$, $p = 10$ and $p = 0.5k$.

We consider three different classes of test matrices, which are summarized below:
\begin{enumerate}
    \item $\mathtt{CIFAR10}$: The CIFAR-10 training set \cite{Krizhevsky2009} consists of $60000$ images of size $32\times 32\times 3$. We choose $10000$ random images, each flattened to a vector, and treat it as a $10000\times 3072 $ data matrix.
    \item $\mathtt{YaleFace64x64}$: Yale face is a full-rank dense matrix of size $165\times 4096$ consisting of $165$ face images each of size $64\times 64$. The flattened image vectors are centered and normalized such that the mean is zero and the entries lie within $[-1,1]$.
    \item $\mathtt{SNN}$: Random sparse non-negative matrices are test matrices used in \cite{SorensenEmbree2016,VoroninMartinsson2017} that is given by,
    \begin{equation*}
        \mathtt{SNN} = \sum_{j = 1}^r s_j x_j y_j^T,
    \end{equation*} where $s_1\geq \cdots \geq s_r >0$ and $x_j\in \R^m, y_j\in \R^n$ are random sparse vectors with non-negative entries. We take $m = 100000, n = 300$ with
        \begin{equation*}
            \sum_{j = 1}^{50}\frac{2}{j}x_j y_j^T + \sum_{j = 51}^{300}\frac{1}{j}x_j y_j^T,
        \end{equation*}
    where the sparse vectors $x_j$'s and $y_j$'s are computed in MATLAB using the command $\mathtt{sprand}(m,1,0.025)$ and $\mathtt{sprand}(n,1,0.025)$, respectively.
    \end{enumerate}
\begin{figure}[!ht]
\hspace*{-0.5cm}
\subfloat[\centering\texttt{CIFAR10} dataset with pivoting on a random sketch with row oversampling.]{\label{subfig:CIFARCA}\includegraphics[scale = 0.33]{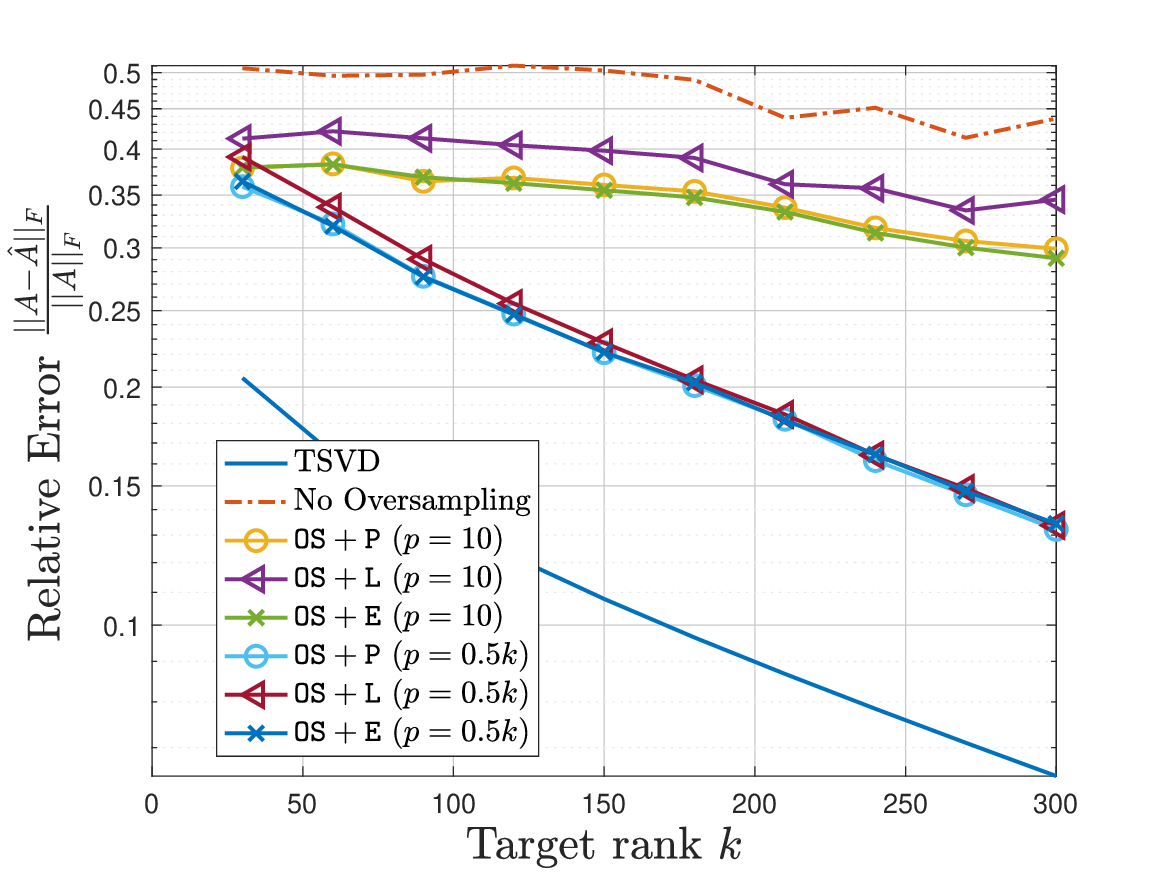}} 
\subfloat[\centering\texttt{CIFAR10} dataset with uniform sampling with row oversampling.]{\label{subfig:CIFARCAuni}\includegraphics[scale = 0.33]{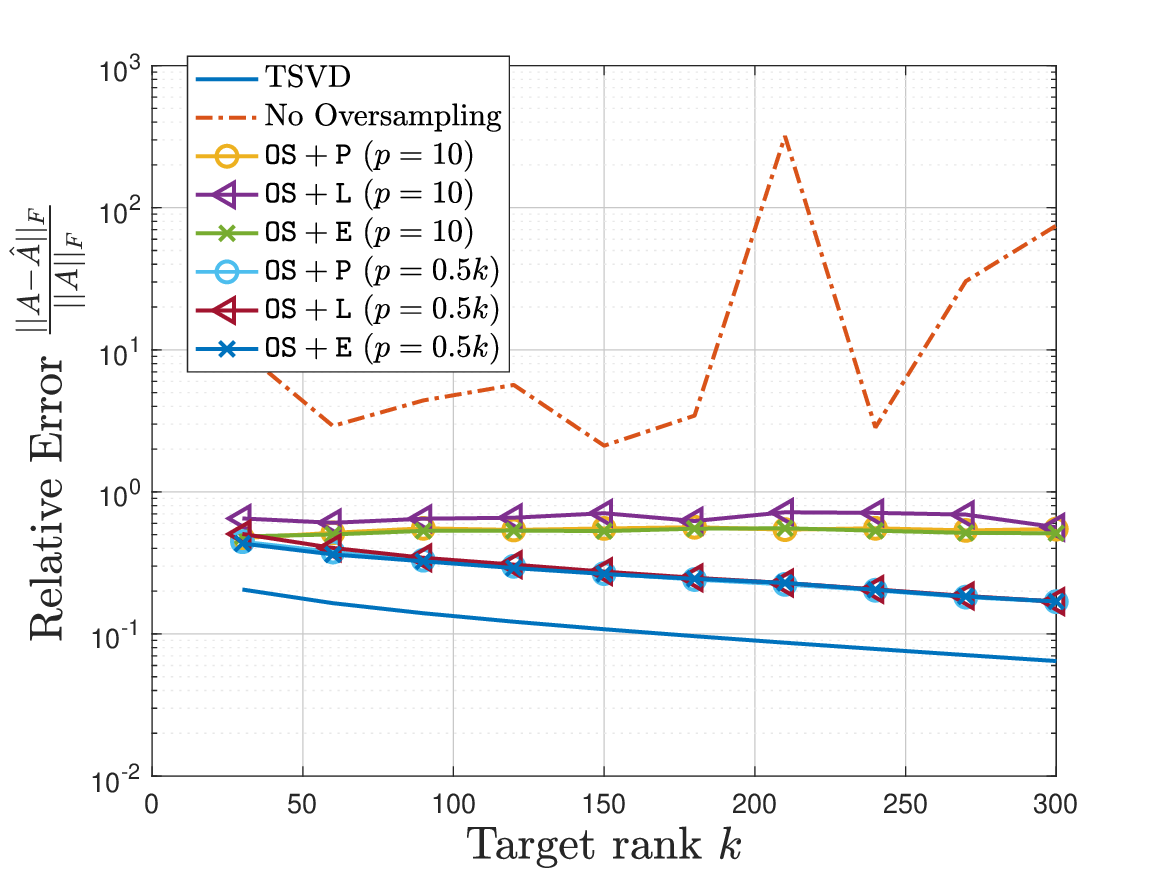}} \\
\hspace*{-0.5cm}
\subfloat[\centering \texttt{YaleFace64x64} dataset with pivoting on a random sketch with row ($p = 0.5k$) and column ($p = 10$) oversampling.]{\label{subfig:YaleCA}\includegraphics[scale = 0.33]{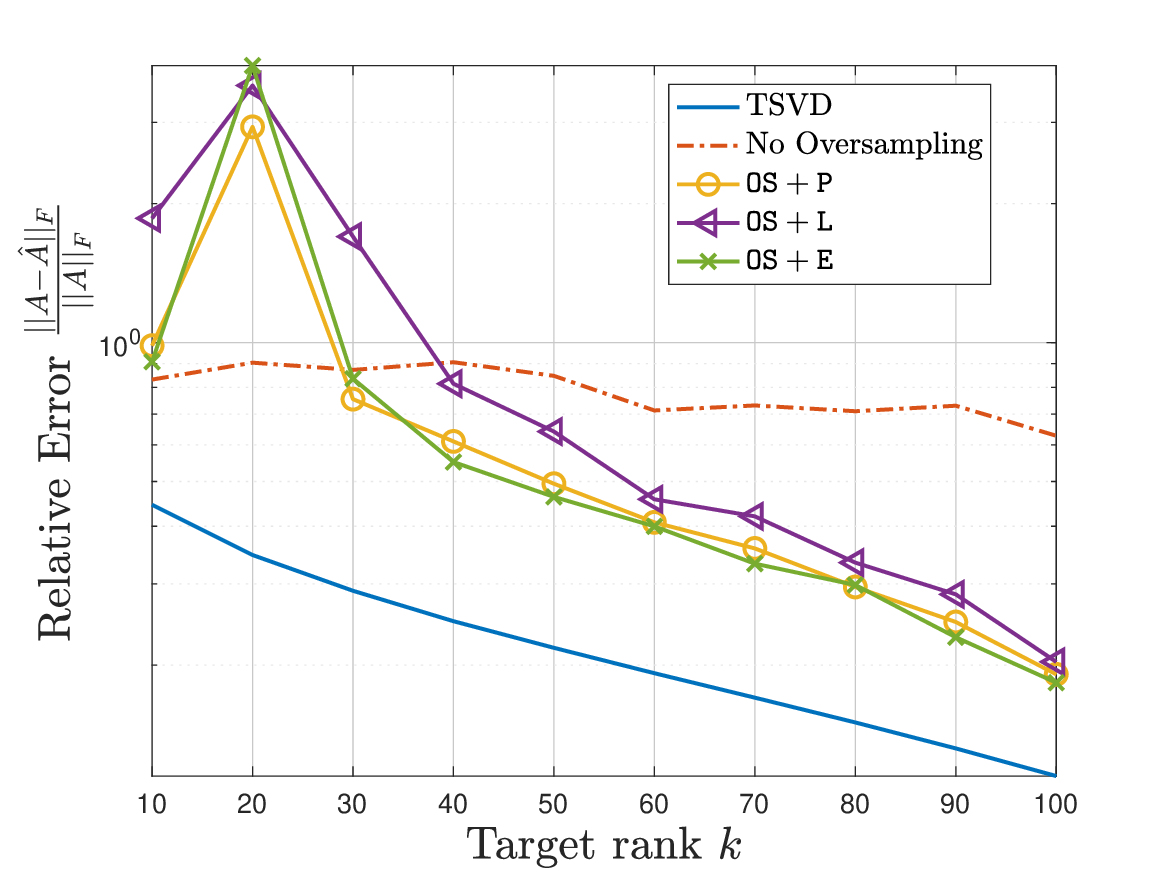}}
\subfloat[\centering \texttt{SNN} dataset with pivoting on a random sketch with row oversampling]{\label{subfig:SNN2CA}\includegraphics[scale = 0.33]{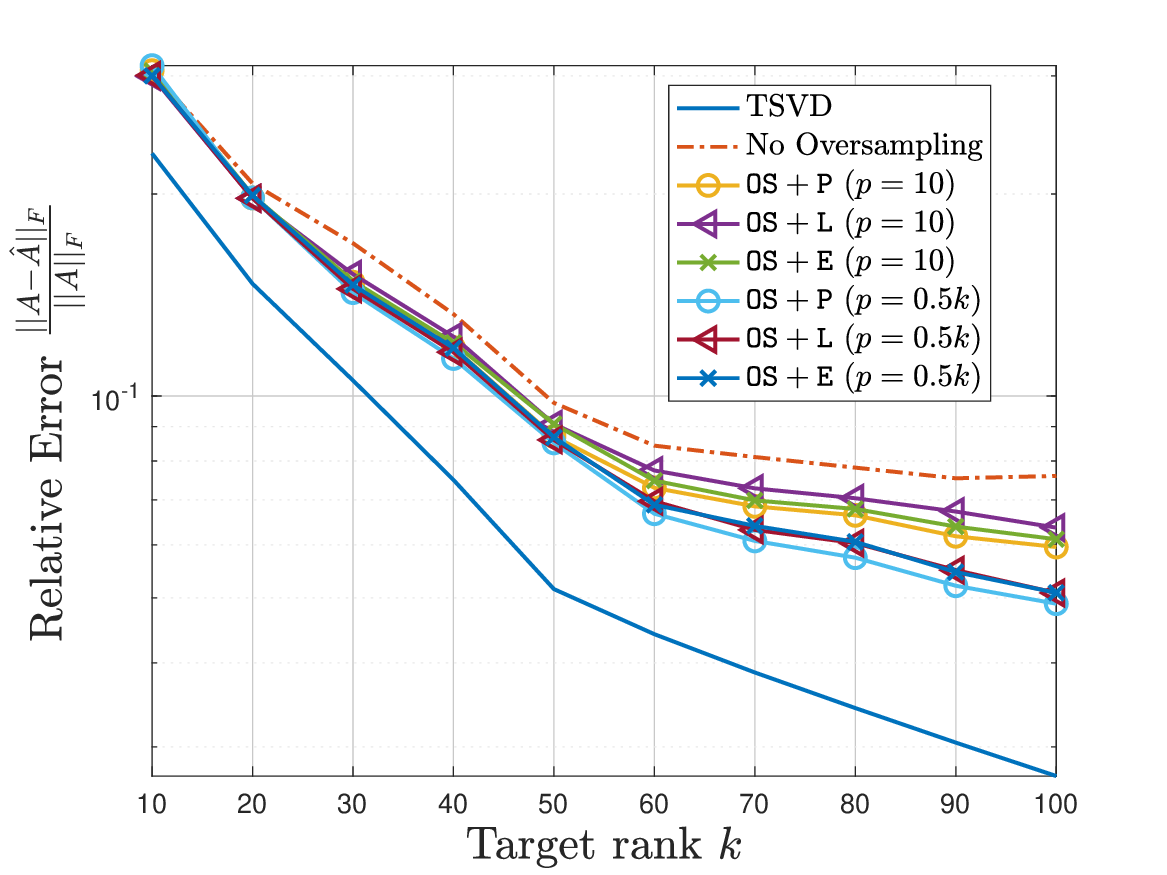}}
\centering 
\caption{Comparison of different oversampling methods for various test matrices. We use pivoting on a random sketch to obtain the initial set of indices except for Figure \ref{subfig:CIFARCAuni} where we use uniform sampling. Then we oversample the row indices using the three oversampling algorithms, $\mathtt{OS\!+\!P}$, $\mathtt{OS\!+\!L}$ and $\mathtt{OS\!+\!E}$. We also oversample column indices in Figure \ref{subfig:YaleCA} to demonstrate that oversampling both row and column indices can be harmful.}
\label{fig:CURCA}
\end{figure}

The results are depicted in Figure \ref{fig:CURCA}. We compare different oversampling algorithms for the three test matrices. We use pivoting on a random sketch (Algorithm \ref{alg:pivsketch}) to obtain the initial set of indices except for Figure \ref{subfig:CIFARCAuni} where we use uniform sampling; pivoting on a random sketch usually results in a good set of indices, while uniform sampling can give a poor set of indices. Then we oversample the row indices using the three oversampling algorithms, $\mathtt{OS\!+\!P}$, $\mathtt{OS\!+\!L}$ and $\mathtt{OS\!+\!E}$. We also oversample column indices in Figure \ref{subfig:YaleCA}. For Figure \ref{subfig:YaleCA}, $p = 0.5k$ for row oversampling and $p = 10$ for column oversampling.

We begin with the top two plots involving the \texttt{CIFAR10} dataset, Figures \ref{subfig:CIFARCA} and \ref{subfig:CIFARCAuni}. We observe that as the oversampling parameter increases, the accuracy improves. The different oversampling techniques yield a similar result with $\mathtt{OS\!+\!L}$ being usually slightly worse. In Figures \ref{subfig:CIFARCA} and \ref{subfig:CIFARCAuni}, we observe that oversampling plays a big role for the accuracy of the CURCA. In Figure \ref{subfig:CIFARCA}, $p = 10$ improves the accuracy slightly, but when the oversampling parameter $p$ becomes proportional to the target rank, we make further progress and we approximately capture the singular value decay rate. In Figure \ref{subfig:CIFARCAuni}, when the initial set of indices are worse, as it can be when we perform uniform sampling, we observe that oversampling makes the approximation more robust even when the CURCA without oversampling yields unstable results. Therefore, oversampling helps the CURCA yield more accurate and stable results. 

In Figure \ref{subfig:YaleCA}, using the \texttt{YaleFace64x64} dataset, we show what happens when we oversample both row and column indices. In this experiment, we oversample the row indices by $p =0.5k$ and the column indices by $p = 10$. In Figure \ref{subfig:YaleCA}, the CURCA can get worse with oversampling and the approximation may become unstable when both column and row indices are oversampled. This is caused by the core matrix $A(I\cup I_0, J\cup J_0)$ underestimating the singular values of $A$ as oversampling in only one of row or column indices improve the condition number of the core matrix, but when we oversample both row and column indices, condition number of the core matrix can become worse. This is similar to the phenomenon happening in Section \ref{subsec:chooseRCdep}, where the CURCA can yield poor accuracy when we choose the rows and columns independently. Therefore, we advocate oversampling in only one of row or column indices for the CURCA. Note that if we do not oversample column indices in Figure \ref{subfig:YaleCA}, so that we only oversample rows for the \texttt{YaleFace64x64} dataset, the relative error decreases steadily, as observed in the other figures in Figure \ref{fig:CURCA}. In Figure \ref{subfig:SNN2CA}, we use the $\mathtt{SNN}$ dataset. The effect of oversampling is less immediate in this example, but the accuracy still improves and the decay rate is more in line with the spectral decay.

Lastly, in all the experiments in Figure \ref{fig:CURCA}, we observe that our algorithm for oversampling $\mathtt{OS+P}$ (Algorithm \ref{alg:oversample_parameter}) is competitive with $\mathtt{OS+E}$ with $\mathtt{OS+L}$ usually being slightly worse. Therefore, Algorithm \ref{alg:oversample_parameter}, which runs with complexity $\bigO(nk^2+nkp)$ is competitive with $\mathtt{OS\!+\!E}$, which run with complexity $\bigO(nk^2p+k^4)$.

\section{Conclusion}
In this work, we study the accuracy and stability of the CUR decomposition with oversampling. We prove the relative norm bounds for the CUR decomposition, $A\approx CU^\dagger R$, and its stabilized version, $A \approx CU_{\epsilon}^\dagger R$. We further show that the stabilized version satisfies a similar relative norm bound in the presence of roundoff errors under the assumption that the rows and columns for the CUR decomposition are chosen reasonably. This means that the rows should be selected based on the chosen columns or vice versa; see Section \ref{subsec:chooseRCdep}. This aims to reduce the quantity $\norm{Q_C(I,:)^{-1}}_2$ (see Theorem \ref{thm:stabresult}), which can be further reduced by oversampling the row indices. For a stable implementation of the CURCA, $A\approx CU^\dagger R$, we recommend the MATLAB implementation
\begin{equation*}
    \mathtt{A_{IJ}^{(3)} = \left(A(:,J)*V/S\right)*\left(W'*A(I,:)\right)} \text{ where } \mathtt{[W,S,V] = svd(U,`econ')}
\end{equation*} or the corresponding version where QR factorization is used instead of the SVD.

We also proposed how oversampling should be done. Oversampling should be done such that it increase the minimum singular value of a certain square matrix that is a submatrix of an orthonormal matrix; see Section \ref{sec:method}. We suggest doing so through projecting the unchosen rows of an orthonormal matrix onto the trailing singular subspace of the square matrix and finding the important unchosen indices that will enrich the trailing singular subspace; see Algorithm \ref{alg:oversample_parameter}. Oversampling improves the stability as the core matrix becomes rectangular and rectangular matrices are more well-conditioned than square matrices. We recommend oversampling in one only one of row or column indices, but not both (see Figure \ref{subfig:YaleCA}) and choose the oversampling parameter to be proportional to the target rank when possible. Experiments show that the algorithm is competitive with other existing methods. Therefore, we advocate oversampling for the accuracy and stability of the CUR decomposition whenever possible.

\bibliographystyle{siamplain}
\bibliography{references}

\appendix

\section{Analysis of the CURBA} \label{app:CURBA}
In this section, we analyze the CURBA, 
\begin{equation}
    A_{IJ}^{(BA)} = A(:,J)\left(A(:,J)^\dagger A A(I,:)^\dagger\right) A(I,:) = C(C^\dagger A R^\dagger) R
\end{equation} with oversampling. For the CURBA, there exists a numerically stable algorthm given by the StableCUR algorithm in \cite{AndersonDuMahoneyMelgaardWuGu2015}. The algorithm computes the CURBA in the following way in MATLAB,
\begin{equation}
    \mathtt{[Q_C,\sim] = qr(A(:,J),0)}, \; \mathtt{[Q_R,\sim ] = qr(A(I,:)',0)}, \; A_{IJ}^{(BA)} = \mathtt{Q_C*(Q_C'*A*Q_R)*Q_R'}.
\end{equation} We use this implementation of the CURBA in the experiments at the end of this section.

We first prove a relative norm bound for the CURBA with oversampling. We make the following standard assumptions, which is analogous to the CURCA counterpart in Section \ref{sec:analysis}.
\begin{assumption} \label{assumptionBA} \*
    \begin{enumerate}
        \item $|I| = |J| = k$ where $k$ is the target rank. The oversampling indices are $I_0$ with $|I_0| = p_1$ for the rows and $J_0$ with $|J_0| = p_2$ for the columns and they satisfy $k+\max\{p_1,p_2\} \leq \rk(A)$.
        \item $A(:,J\cup J_0)$ has full column rank and $A(I\cup I_0,:)$ has full row rank.
        \item $X(:,J) \in \R^{k\times (k+p_2)}$ has full row rank, where $X \in \R^{k\times n}$ is a row space approximator of $A$.
        \item $Y(I,:) \in \R^{k\times (k+p_2)}$ has full row rank, where $Y\in \R^{m\times k}$ is a column space approximator of $A$.
    \end{enumerate}
\end{assumption}
The assumption that $X(:,J)$ and $A(I,:)$ having a full row rank and $Y(I,:)$ and $A(:,J)$ having a full column rank is generic, since most methods that pick good row and column indices satisfy this assumption. If one of $X$ or $Y$ are not available then $A(I,:)$ or $A(:,J)$ can be used instead, respectively.

We now prove the result for the CURBA with oversampling. The results shown below is a simple extension of \cite[Lemma 4.2]{SorensenEmbree2016} and \cite[Theorem 1]{DongMartinsson2023}. Lemma \ref{lemma:BAO} considers one-sided projection with oversampling where we project $A$ onto the chosen columns of $A$. In Theorem \ref{thm:CURBA}, we consider the CURBA, $C(C^\dagger AR^\dagger) R$ where we project $A$ onto the chosen rows and columns of $A$.

\begin{lemma} \label{lemma:BAO}
    Under Assumption \ref{assumptionBA},
    \begin{equation}
        \norm{A-CC^\dagger A} \leq \norm{Q_X(J\cup J_0,:)^\dagger}_2 \norm{A-AX^\dagger X}
    \end{equation} where $\norm{\cdot}$ is any unitarily invariant norm and $Q_X\in \R^{n\times k}$ is an orthonormal matrix spanning the columns of $X^T$.
\end{lemma}
\begin{proof}
For shorthand let $J_* = J\cup J_0$ and let $\Pi_{J_*} = I_n(:,J_*) \in \R^{n\times (k+p)}$ such that $C = A(:,J_*) = A\Pi_{J_*}$. We first define two oblique projectors
    \begin{equation*}
        \mathcal{P}_X := \Pi_{J_*} (X\Pi_{J_*})^\dagger X, \hspace{0.5cm} \mathcal{P}_C := \Pi_{J_*} C^\dagger A \in \R^{n\times n}.
    \end{equation*} Note that since $C$ has full column rank, $C^\dagger A\Pi_{J_*} = C^\dagger C = I_{k+p}$, and 
    \begin{equation*}
        \mathcal{P}_C \mathcal{P}_X = \Pi_{J_*} C^\dagger A\Pi_{J_*} (X\Pi_{J_*})^\dagger X = \mathcal{P}_X.
    \end{equation*}
    Therefore we get
    \begin{equation*}
        A-CC^\dagger A = A(I-\mathcal{P}_C) = A(I-\mathcal{P}_C)(I-\mathcal{P}_X) = (I-CC^\dagger )A(I-\mathcal{P}_X).
    \end{equation*} Since $X\Pi_{J_*} = X(:,J_*)$ has full row rank,
    \begin{equation*}
        X\mathcal{P}_X = X\Pi_{J_*}\left(X\Pi_{J_*}\right)^\dagger X = X
    \end{equation*} and we obtain
    \begin{equation*}
        (I-\mathcal{P}_X) = (I-X^\dagger X)(I-\mathcal{P}_X).
    \end{equation*} Now putting these together, we obtain
    \begin{align*}
        \norm{A-CC^\dagger A} &= \norm{(I-CC^\dagger) A(I-X^\dagger X)(I-\mathcal{P}_X)} \\
        &\leq \norm{I-CC^\dagger}_2 \norm{A(I-X^\dagger X)}\norm{I-\mathcal{P}_X}_2 \\
        &= \norm{I-\mathcal{P}_X}_2 \norm{A(I-X^\dagger X)},
    \end{align*} since $\norm{I-CC^\dagger}_2 = 1$ as $CC^\dagger$ is an orthogonal projector. Now since $\mathcal{P}_X$ is an oblique projector \cite{Szyld2006}, $\norm{I-\mathcal{P}_X}_2 = \norm{\mathcal{P}_X}_2$ so the result follows by noting that
    \begin{equation*}
        \norm{\mathcal{P}_X}_2 = \norm{(X\Pi_J)^\dagger X}_2 = \norm{Q_X(J,:)}_2.
    \end{equation*}
\end{proof}
\begin{theorem} \label{thm:CURBA}
        Under Assumption \ref{assumptionBA},
    \begin{align} 
        \norm{A-A_{I\cup I_0,J\cup J_0}^{(BA)}} &\leq \norm{Q_X(J \cup J_0,:)^\dagger}_2 \norm{A-AX^\dagger X} \label{eq:CURBAbound} \\  & \hspace{1cm} + \norm{Q_Y(I\cup I_0,:)^\dagger}_2 \norm{A-YY^\dagger A} \nonumber
    \end{align} where $\norm{\cdot}$ is any unitarily invariant norm and $Q_X\in \R^{n\times k}$ and $Q_Y\in \R^{m\times k}$ are the orthonormal matrices spanning the columns of $X^T$ and $Y$, respectively.
\end{theorem}
\begin{proof}
    The proof follows by Lemma \ref{lemma:BAO} and noting that 
    \begin{align*}
        \norm{A-CC^\dagger A R^\dagger R} &\leq \norm{A-CC^\dagger A} +\norm{CC^\dagger A-CC^\dagger AR^\dagger R} \\
        &\leq \norm{A-CC^\dagger A} +\norm{CC^\dagger}_2\norm{A-AR^\dagger R} \\
        &= \norm{A-CC^\dagger A} +\norm{A-AR^\dagger R} \\
        &\leq \norm{Q_X(J,:)^\dagger}_2 \norm{A-AX^\dagger X} + \norm{Q_Y(I,:)^\dagger}_2 \norm{A-YY^\dagger A}
    \end{align*} where the inequality for the second term $\norm{A-AR^\dagger R}$ in the final line can be shown by considering $A^T$ in Lemma \ref{lemma:BAO}.
\end{proof}
\begin{remark}\* \label{remark1}
    \begin{enumerate}
        \item The difference between Theorem \ref{thm:CURBA} and its counterpart without oversampling, e.g. \cite[Theorem 1]{DongMartinsson2023}, are the terms $\norm{Q_X(J,:)^\dagger}_2$ and $\norm{Q_Y(I,:)^\dagger}_2$ where instead of the matrix inverse, we have the pseudoinverse. This tightens the bound \eqref{eq:CURBAbound} as 
        \begin{equation*}
            \norm{Q_X(J\cup J_0,:)^\dagger}_2 \leq \norm{Q_X(J,:)^{-1}}_2,
        \end{equation*} and
        \begin{equation*}
            \norm{Q_Y(I\cup I_0,:)^\dagger}_2 \leq \norm{Q_Y(I,:)^{-1}}_2,
        \end{equation*}
        where $I_0$ and $J_0$ are the extra indices for oversampling.
        \item There are many possible choices for $X$ and $Y$. For example, if we choose them to be the dominant right and left singular vectors of $A$, then we get the bound similar to the one in \cite{SorensenEmbree2016} involving the best rank-$k$ approximation since $\norm{A-AX^\dagger X} = \norm{A-YY^\dagger A} = \norm{A-\lowrank{A}{k}}$ where $\lowrank{A}{k}$ is the best rank-$k$ approximation to $A$. If we choose $X$ and $Y$ to be the row sketch and the column sketch (see Algorithm \ref{alg:pivsketch}) then we involve the randomized SVD error \cite{hmt}. Choosing $X$ and $Y$ to be the dominant singular vectors in Theorem \ref{thm:CURBA} can be the optimal choice, but when $X$ and $Y$ are approximators, which can be computed easily, the bound \eqref{eq:CURBAbound} can be computed a posteriori.
        \item When $k+\min\{p_1,p_2\} \geq \rk(A)$ and $\rk(C) = \rk(R) = \rk(A)$, we have $A = CC^\dagger AR^\dagger R$ \cite{HammHuang2020}.
    \end{enumerate}
\end{remark}

Theorem \ref{thm:CURBA} suggests that oversampling helps to improve the accuracy of the CUR decomposition $A_{IJ}^{(BA)}$. However, the numerical simulations shown below illustrate that the improvement is not too significant. 
\begin{figure}[!ht]
\hspace*{-0.5cm}
\subfloat[\centering \texttt{CIFAR10} dataset with pivoting on a random sketch with row oversampling.]{\label{subfig:CIFARBA}\includegraphics[scale = 0.32]{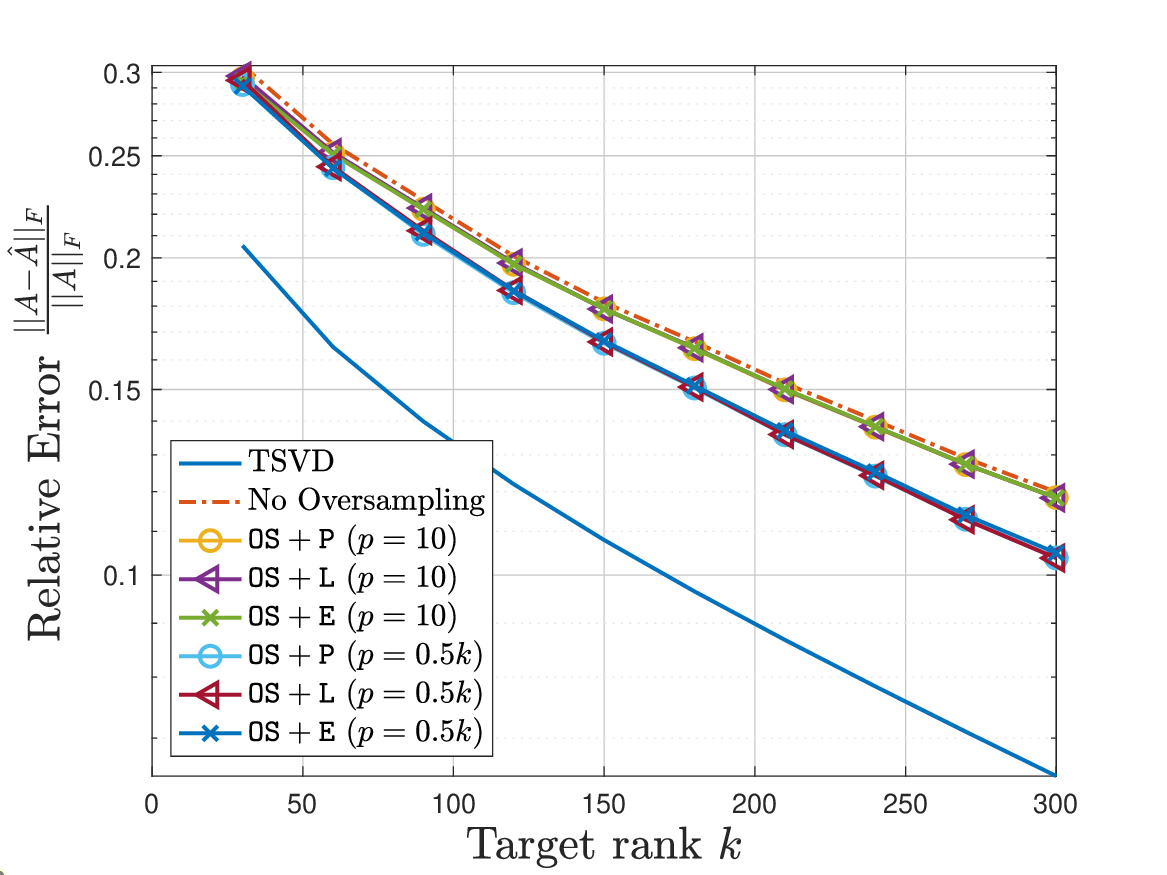}}
\subfloat[\centering \texttt{YaleFace64x64} dataset with pivoting on a random sketch with row $(p=0.5k)$ and column $(p = 10)$ oversampling.]{\label{subfig:YaleBA}\includegraphics[scale = 0.32]{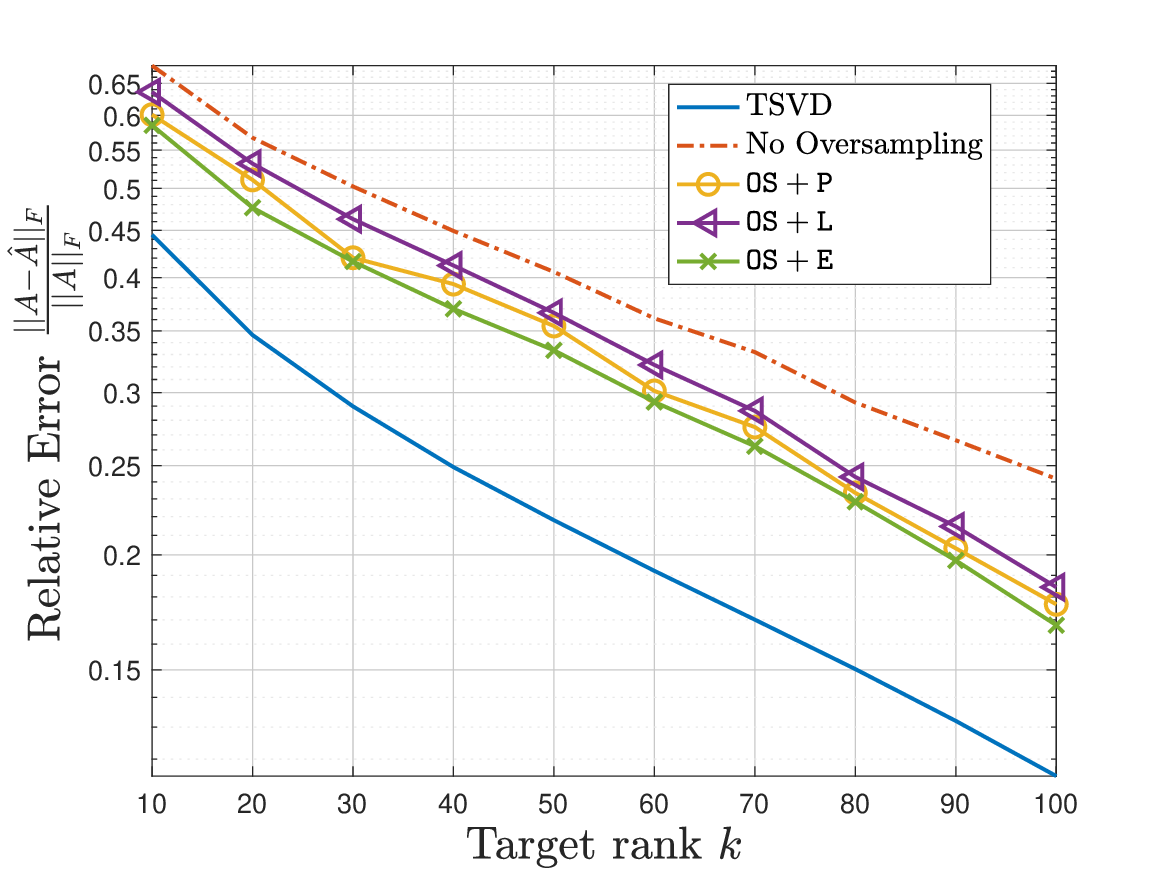}} 
\centering 
\caption{The effect of oversampling for the CURBA. We use pivoting on a random sketch (Algorithm \ref{alg:pivsketch}) to obtain the initial set of indices. Then we oversample the row indices using the three oversampling algorithms, $\mathtt{OS\!+\!P}$, $\mathtt{OS\!+\!L}$ and $\mathtt{OS\!+\!E}$. We also oversample column indices in Figure \ref{subfig:YaleBA}.}
\label{fig:CURBA}
\end{figure}

The numerical results are depicted in Figure \ref{fig:CURBA}, illustrating the effect of oversampling for the CURBA. We use pivoting on a random sketch (Algorithm \ref{alg:pivsketch}) to obtain the initial set of indices. Then we oversample the row indices using the three oversampling algorithms, $\mathtt{OS\!+\!P}$, $\mathtt{OS\!+\!L}$ and $\mathtt{OS\!+\!E}$ as in Section \ref{subsec:oversampleNI}. We also oversample column indices in Figure \ref{subfig:YaleBA}.

We start with the \texttt{CIFAR10} dataset in Figure \ref{subfig:CIFARBA}. We first observe that as the oversampling parameter increases, the accuracy improves and the different oversampling techniques yield a similar result. However, there is no significant improvement when oversampling is used for the CURBA. This shows that the CURBA gives a robust approximation and oversampling only plays a slight role in its accuracy. In Figure \ref{subfig:YaleBA}, we used the \texttt{YaleFace64x64} dataset to illustrate the effect of oversampling both rows and columns. We observe that when we oversample both rows and columns, we obtain a higher accuracy for the CURBA. This is expected as we are enlarging the subspace that we project onto $A$, which increases the accuracy of the CURBA. This is contrary to the CURCA, as the CURCA can become more inaccurate and unstable when we oversample both rows and columns; see Section \ref{subsec:oversampleNI}.

To conclude, while oversampling improves the accuracy of the CURBA, it has less significant effect on its accuracy and stability. 

\section{Stability of rank-deficient systems} \label{app:stab}
In this section, we prove the stability result for solving rank-deficient underdetermined linear systems (Theorem \ref{thm:rankdef}). The result can also be extended to rank-deficient overdetermined problems; see Appendix \ref{subsec:Overdetermined}. Theorem \ref{thm:rankdef} was used to derive an expression for the computed solution to
\begin{equation}
    \min_x\norm{(\hat{U}^T)_\epsilon x-[\hat{C}]_i^T}_2
\end{equation} for each row of $\hat{C}$. The statement and the proof is shown below. The proof follows a similar method outlined in \cite[Section 4.1]{Nakatsukasa2020}.

\begin{theorem}[Theorem \ref{thm:rankdef}]
    Consider the (rank-deficient) underdetermined linear system,
    \begin{equation} \tag{\ref{eq:rankdefeq}}
        \min_{x'} \norm{B_\epsilon x'-b}_2 
    \end{equation} where $B_\epsilon \in \R^{m\times n}$ $(m\leq n)$ is (possibly) rank-deficient $(\rk(B_\epsilon)\leq m)$ with singular vales larger than $\epsilon$ and $b\in \R^{m}$. Then assuming $\epsilon > \gamma\norm{B_\epsilon}_2$, the minimum norm solution to \eqref{eq:rankdefeq} can be computed in a backward stable manner, i.e., the computed solution $\hat{s}$ satisfies
    \begin{equation*}
        \hat{s} = \left( B_\epsilon + E_1 \right)^\dagger \left(b+E_2\right)
    \end{equation*} where $\norm{E_1}_2 \leq \gamma \norm{B_\epsilon}_2$ and $\norm{E_2}_2\leq \gamma\norm{b}_2$.
\end{theorem}
\begin{proof}
    First, if $B_\epsilon$ has full row-rank then the statement follows by the stability result in \cite[Theorem 21.4]{HighamStabilityBook}. If $B_\epsilon$ is rank-deficient,
    the stability result in \cite[Theorem 21.4]{HighamStabilityBook} cannot be invoked as it is only applicable for underdetermined full-rank linear systems. So we project $B_\epsilon$ onto its column space $Q_\epsilon$ to make the problem an underdetermined full-rank linear system:
    \begin{equation*}
        \min_{x'}\norm{Q_\epsilon^TB_{\epsilon}x'-Q_\epsilon^Tb}_2.
    \end{equation*}

    Let $\hat{Q}_\epsilon $ be the computed column space of $B_\epsilon$. Then $\hat{Q}_\epsilon$ is the exact column space of $B_\epsilon + \Delta B$ where $\norm{\Delta B}_2 \leq \gamma \norm{B_\epsilon}_2$, $\hat{Q}_\epsilon^TB_\epsilon$ is numerically full-rank with singular values larger than $\epsilon$ and $\hat{Q}_\epsilon \hat{Q}_\epsilon^TB_\epsilon = B_\epsilon+ E$ where $\norm{E}_2 \leq \gamma\norm{B_\epsilon}_2$ \cite[Chapter 5.4.1]{GolubVanLoan2013}. Note the following from matrix-matrix (or matrix-vector) multiplication \cite[Section 3.5]{HighamStabilityBook},
\begin{equation*}
    fl(\hat{Q}_\epsilon^T B_\epsilon) = \hat{Q}_\epsilon^TB_\epsilon + E^{(1)}
\end{equation*} where $\norm{E^{(1)}}_2\leq \gamma\norm{B_\epsilon}_2$ and
\begin{equation*}
    fl(\hat{Q}_\epsilon^T b) = \hat{Q}_\epsilon^Tb + E^{(2)}
\end{equation*} where $\norm{E^{(2)}}_2\leq \gamma\norm{b}_2$. Since $\hat{Q_\epsilon^TB_\epsilon} \in \R^{\rk(B_\epsilon) \times n}$ is a fat rectangular matrix, we solve the following underdetermined full-rank linear system,
\begin{equation} \label{eq:udprobmod}
   \min_{x'} \norm{(\hat{Q}_\epsilon^TB_\epsilon)x'-\hat{Q}_\epsilon^Tb}_2.
\end{equation}
Assuming $\gamma\kappa_2(B_\epsilon) \leq \gamma\norm{B_\epsilon}_2/\epsilon  <1$ (where $\kappa_2$ denotes the 2-norm condition number), we obtain the computed solution of \eqref{eq:udprobmod}, $\hat{s}$, satisfying \cite[Theorem 21.4]{HighamStabilityBook},
\begin{align*}
    \hat{s} = fl\left(\left(\hat{Q}_\epsilon^TB_\epsilon\right)^\dagger \hat{Q}_\epsilon^Tb \right) &= \left(fl(\hat{Q}_\epsilon^TB_\epsilon)+E^{(3)}\right)^\dagger fl(\hat{Q}_\epsilon^Tb) \\
    &= \left(\hat{Q}_\epsilon^T B_\epsilon + E^{(1)}+E^{(3)}\right)^\dagger \left(\hat{Q}_\epsilon^Tb+E^{(2)}\right) \\
    &= \left(\hat{Q}_\epsilon \hat{Q}_{\epsilon}^TB_\epsilon + \hat{Q}_\epsilon E^{(1)}+\hat{Q}_\epsilon E^{(3)}\right)^\dagger \left(b+\hat{Q}_\epsilon E^{(2)}\right) \\
    &= \left(B_\epsilon + E + \hat{Q}_\epsilon E^{(1)}+\hat{Q}_\epsilon E^{(3)}\right)^\dagger \left(b+\hat{Q}_\epsilon E^{(2)}\right)
\end{align*} where $\norm{E^{(3)}}_2 \leq \gamma\norm{B_{\epsilon}}_2$ and we use $(QA)^\dagger b = A^\dagger Q^Tb$ for a tall-skinny orthonormal matrix $Q$ in line $3$. Therefore, assuming that $\gamma \norm{B_\epsilon}_2 < \epsilon$, the computed solution of \eqref{eq:rankdefeq}, $\hat{s}$ satisfies
\begin{equation} \label{eq:lsnumstab}
    \hat{s} = fl(B_\epsilon^\dagger b) = \left(B_\epsilon + E_1\right)^\dagger \left(b+E_2 \right),
\end{equation} where $\norm{E_1}_2 \leq \gamma \norm{B_\epsilon}_2$ and $\norm{E_2}_2\leq \gamma \norm{b}_2$.
\end{proof}

\subsection{Extension to rank-deficient overdetermined problems} \label{subsec:Overdetermined}
Theorem \ref{thm:rankdef} can be extended to include rank-deficient overdetermined least-sqaures problems. Consider the rank-deficient overdetermined least-squares problem,
\begin{equation*} 
    \min_{x'} \norm{C_\epsilon x'-b}_2 
\end{equation*} where $C_\epsilon \in \R^{m\times n}$ $(m\geq n)$ is a rank-deficient $(\rk(C_\epsilon)< n)$, tall-skinny matrix with singular vales larger than $\epsilon$ and $b\in \R^{m}$. Then under the same assumption as Theorem \ref{thm:rankdef}, i.e., $\epsilon > \gamma \norm{B_\epsilon}_2$, the solution to the overdetermined least-squares problem can be computed in a backward stable manner.

The proof proceeds in the same way as Theorem \ref{thm:rankdef}, as once the tall-skinny matrix $C_\epsilon$ is projected onto its column space $W_\epsilon$, $W_\epsilon^T C_\epsilon \in \R^{\rk(C_\epsilon)\times n}$ becomes a fat full-rank matrix, and we are now solving the underdetermined full-rank linear system: $\min\limits_{x'} \norm{(\hat{W}_\epsilon^TC_\epsilon)x'-\hat{W}_\epsilon^Tb}_2$.

\section{Two lemmas} \label{app:lemmas}
We give the proofs for the two lemmas, Lemma \ref{lemma:NB1} and Lemma \ref{lemma:NB2} in this section. 
\begin{lemma}[Lemma \ref{lemma:NB1}]
    Under Assumption \ref{assumptions}, for any $\Delta C$ and $\Delta U$,
    \begin{equation}
        \norm{(C+\Delta C)(U+\Delta U)_\epsilon^\dagger}_2 \leq \norm{Q_C(I_*,:)^\dagger}_2 \left(1 + \frac{1}{\epsilon}\norm{\Delta U}_2 \right) + \frac{1}{\epsilon}\norm{\Delta C}_2.
    \end{equation}
\end{lemma}
\begin{proof}
    Let $C = Q_C R_C $ be the thin QR factorization of $C$. Then
    \begin{align*}
        \norm{(C+\Delta C)(U+\Delta U)_\epsilon^\dagger}_2 &\leq \norm{Q_CR_C(U+\Delta U)_\epsilon^\dagger}_2 + \norm{\Delta C (U+\Delta U)_\epsilon^\dagger}_2 \\
        &= \norm{R_C(U+\Delta U)_\epsilon^\dagger}_2 + \norm{\Delta C (U+\Delta U)_\epsilon^\dagger}_2 \\
        &= \norm{(\Pi_{I_*}^T Q_C)^\dagger \Pi_{I_*}^T Q_C R_C(U+\Delta U)_\epsilon^\dagger}_2 +  \norm{\Delta C (U+\Delta U)_\epsilon^\dagger}_2 \\
        &\leq \norm{(\Pi_{I_*}^T Q_C)^\dagger}_2 \norm{U (U+\Delta U)_\epsilon^\dagger}_2 + \frac{1}{\epsilon}\norm{\Delta C}_2
    \end{align*} where we use $(\Pi_{I_*}^T Q_C)^\dagger \Pi_{I_*}^T Q_C = I$ for the penultimate line.
    The result follows by noting that $\norm{U (U+\Delta U)_\epsilon^\dagger}_2$ simplifies to 
    \begin{align*}
    \norm{U (U+\Delta U)_\epsilon^\dagger}_2 = \norm{(U+\Delta U) (U+\Delta U)_\epsilon^\dagger - \Delta U (U+\Delta U)_\epsilon^\dagger}_2 \leq  1 + \frac{1}{\epsilon}\norm{\Delta U}_2.
    \end{align*}
\end{proof}

\begin{lemma}[Lemma \ref{lemma:NB2}]
    Under Assumption \ref{assumptions}, for any $\Delta C$ and $\Delta U$,
    \begin{equation}
        (C+\Delta C)(U+\Delta U)_\epsilon^\dagger R\Pi_J = C + E_*
    \end{equation} where 
    \begin{align*}
         \norm{E_*}_2 \leq \norm{(Q_C(I_*,:)^\dagger}_2 \left(\epsilon  + 2\norm{\Delta U}_2 + \frac{1}{\epsilon}\norm{\Delta U}_2^2\right) +  \norm{\Delta C}_2 \left(1+\frac{\norm{\Delta U}_2}{\epsilon}\right).
     \end{align*}
\end{lemma}
\begin{proof}
    We divide the expression into two pieces:
    \begin{align*}
         (C+\Delta C)(U+\Delta U)_\epsilon^\dagger R\Pi_J  =  \underbrace{C(U+\Delta U)_\epsilon^\dagger R\Pi_J}_{(i)}+ \underbrace{\Delta C (U+\Delta U)_\epsilon^\dagger R \Pi_J}_{(ii)} 
    \end{align*} and treat them separately. Let the thin SVD of $U+\Delta U$ be 
    \begin{equation*}
        U+\Delta U = W\Sigma V^T = [W_1,W_2]\begin{bmatrix}
            \Sigma_1 & \\ & \Sigma_2
        \end{bmatrix}[V_1,V_2]^T
    \end{equation*}
     where $\Sigma_2$ contains the singular values of $U +\Delta U$ smaller than $\epsilon$ and $C = Q_CR_C$ be the thin QR decomposition of $C$. 
     
     We begin by examining the matrix (i).
     \begin{align*}
         C(U+\Delta U)_\epsilon^\dagger R\Pi_J &= C(U+\Delta U)_\epsilon^\dagger U = C(U+\Delta U)_\epsilon^\dagger (U+\Delta U) - C(U+\Delta U)_\epsilon^\dagger \Delta U \\
         &= CV_1V_1^T - C(U+\Delta U)_\epsilon^\dagger \Delta U = C - CV_2V_2^T - C(U+\Delta U)_\epsilon^\dagger \Delta U \\
         &= C + E_1
     \end{align*} where $E_1 = - CV_2V_2^T - C(U+\Delta U)_\epsilon^\dagger \Delta U$ satisfies
     \begin{align*}
         \norm{E_1}_2 &\leq \norm{CV_2V_2^T}_2+\norm{C(U+\Delta U)_\epsilon^\dagger \Delta U}_2 \\
         &\leq \norm{CV_2V_2^T}_2 + \norm{Q_C(I_*,:)^\dagger}_2 \left(1 + \frac{1}{\epsilon}\norm{\Delta U} \right)\norm{\Delta U}_2
     \end{align*} by Lemma \ref{lemma:NB1}. We bound $\norm{CV_2V_2^T}_2$ as in the final part of the proof of Theorem \ref{thm:CAe}.
     \begin{align*}
         \norm{CV_2V_2^T}_2 &= \norm{R_C V_2 V_2^T}_2 = \norm{(\Pi_{I_*}^TQ_C)^\dagger \Pi_{I_*}^TQ_C R_C V_2 V_2^T}_2 = \norm{(\Pi_{I_*}^TQ_C)^\dagger UV_2 V_2^T}_2 \\ &\leq \norm{(\Pi_{I_*}^TQ_C)^\dagger}_2 \norm{(U+\Delta U)V_2V_2^T-\Delta U V_2V_2^T}_2 \\
         &\leq \norm{(\Pi_{I_*}^TQ_C)^\dagger}_2 \left(\norm{W_2\Sigma_2}_2 + \norm{\Delta U}_2 \right) \\
         &\leq \norm{(\Pi_{I_*}^TQ_C)^\dagger}_2 \left(\epsilon  + \norm{\Delta U}_2 \right).
     \end{align*} Therefore, (i) can be bounded as 
     \begin{equation*}
         C(U+\Delta U)_\epsilon^\dagger R\Pi_J = C + E_1
     \end{equation*} where $\norm{E_1}_2 \leq \norm{(Q_C(I_*,:)^\dagger}_2 \left(\epsilon  + 2\norm{\Delta U}_2 + \frac{1}{\epsilon}\norm{\Delta U}_2^2\right)$. 

     Next, we bound the matrix (ii). Let $E_2 = \Delta C (U+\Delta U)_\epsilon^\dagger R \Pi_J$, then
     \begin{align*}
         \norm{E_2}_2 &= \norm{\Delta C (U+\Delta U)_\epsilon^\dagger R \Pi_J}_2 \\
         &\leq \norm{\Delta C}_2 \norm{(U+\Delta U)_\epsilon^\dagger U}_2 \\
         &\leq \norm{\Delta C}_2 \left( \norm{(U+\Delta U)_\epsilon^\dagger (U+\Delta U)}_2 + \norm{(U+\Delta U)_\epsilon^\dagger \Delta U}_2 \right) \\
         &\leq \norm{\Delta C}_2 \left(1+\frac{\norm{\Delta U}_2}{\epsilon}\right).
     \end{align*}

     Putting everything together and setting $E_* = E_1 + E_2$, we get the desired result:
     \begin{equation*} 
         (C+\Delta C)(U+\Delta U)_\epsilon^\dagger R\Pi_J = C + E_*
     \end{equation*} where 
     \begin{align*}
         \norm{E_*}_2 \leq \norm{(Q_C(I_*,:)^\dagger}_2 \left(\epsilon  + 2\norm{\Delta U}_2 + \frac{1}{\epsilon}\norm{\Delta U}_2^2\right) +  \norm{\Delta C}_2 \left(1+\frac{\norm{\Delta U}_2}{\epsilon}\right).
     \end{align*}
\end{proof}

\end{document}